\documentclass[leqno,12pt]{article}
\usepackage [dvips]{graphicx}
\usepackage{epsf}
\usepackage{amssymb,amsmath}
\usepackage{psfrag}
\usepackage{graphicx}

\usepackage{amsfonts}
\newtheorem{theorem}{Theorem}
\newtheorem{proposition}[theorem]{Proposition}
\newtheorem{lemma}[theorem]{Lemma}
\newtheorem{definition}[theorem]{Definition}
\newtheorem{corollary}[theorem]{Corollary}
\newtheorem{example}[theorem]{Example}
\newtheorem{remark}[theorem]{Remark}
\newenvironment{proof}{\noindent{\bf Proof. }}{\par}
\newcommand{\C}{{\mathbb C}}
\newcommand{\R}{{\mathbb R}}
\newcommand{\Z}{{\mathbb Z}}

\renewcommand{\ll}{{\langle}}
\newcommand{\rr}{{\rangle}}

\newcommand{\Res}{{\mathrm{Res}}}

\newcommand{\CL}{{\mathcal{L}}}

\newcommand{\Ker}{\operatorname{Ker}}
\newcommand{\vol}{\operatorname{vol}}

\renewcommand{\c}{\mathfrak{c}}

\title{Paradan's wall crossing formula for partition functions and Khovanski-Pukhlikov differential operator.}
\author{Arzu Boysal and Mich{\`e}le Vergne}
\date{March 2008}
\begin{document}
\maketitle \maketitle {\small \tableofcontents}

\section{Introduction}
The function computing the number of ways one can decompose a vector
as a linear combination with nonnegative integral coefficients of a
fixed finite  set of integral vectors is called a {\bf partition
function.} This  problem can be expressed in terms of polytopes as
follows. Let $A$ be a  $r$ by $N$ integral matrix with column
vectors $ \phi_1,\ldots,\phi_N$, and assume that the elements $\phi_k$ generate
the lattice $\Z^r$.
 Let $a\in \Z^r$ be a  $r$-dimensional
integral column vector and let $P(\Phi,a):=\{ y\in {\R}_{\geq 0}^N |
A y= a \}$ be the convex polytope associated to
$\Phi=[\phi_1,\phi_2,\ldots,\phi_N]$ and $a$. The number of ways one
can decompose $a$ as a linear combination with nonnegative integral
coefficients of the vectors $\phi_i$ is the number of integral
points in $P(\Phi,a)$. The function $a\to |P(\Phi,a)\cap \Z^N|$ will
be called the partition function $k(\Phi)(a)$. It is intuitively
clear that $k(\Phi)(a)$ is related to the volume function
$\vol(\Phi)(a)= {\rm volume}(P(\Phi,a))$.  The latter varies
polynomially as a function of $a$, provided  the polytope
$P(\Phi,a)$ does not change  `shape', that is, when $a$ varies in a
chamber $\c$ for $A$. In short, there is a decomposition of $\R^r$
in closure of  chambers $\c_i$   and polynomial functions $v(\Phi,\c_i)$ such that  the function $\vol(\Phi)(a)$
coincide with the
polynomial  function $v(\Phi,\c_i)(a)$ on each cone $\c_i$. Similarly, there exists
quasi-polynomial functions $k(\Phi,\c_i)$ on $\Z^r$  such that  the function  $k(\Phi)(a)$
coincide with  the quasi-polynomial function
$k(\Phi,\c_i)(a)$ on $\c_i\cap \Z^r$.

When $\c_1$ and $\c_2$ are adjacent chambers, P-E.~Paradan
\cite{Par}  gave a remarkable formula for the quasi-polynomial
function $k(\Phi,\c_1)-k(\Phi,\c_2)$ as a convolution of  distributions.
 There is an analogous
formula for $v(\Phi, \c_1)-v(\Phi,\c_2)$. In this note, we give an
elementary algebraic proof of Paradan's  convolution formula for the jumps.
  We also express  $k(\Phi,\c_1)-k(\Phi,\c_2)$ and  $v(\Phi, \c_1)-v(\Phi,\c_2)$
by  one-dimensional residue formulae.  Let
us describe these residue formulae.

Let $\c_1$ and  $\c_2$  be two adjacent chambers  lying on two sides of a hyperplane $W$
(determined by a primitive vector $E$).  Define $\Phi_0=\Phi \cap
W$. The intersection of $\overline{\c_1}$ and  $\overline{\c_2}$ is contained in the
closure of a chamber $\c_{12}$ of $\Phi_0$.

\begin{theorem}
\begin{itemize}
\item Let $v_{12}=v(\Phi_0,\c_{12})$ be the polynomial function on
$W$ associated to the chamber $\c_{12}$ of $\Phi_0$.  Let $V_{12}$
be any polynomial function on $\R^r$ extending $v_{12}$.  Then, if
$\langle E,\c_1\rangle > 0$, we have for $a\in \R^r$
$$
v(\Phi,\c_1)(a)-v(\Phi,\c_2)(a)=\Res_{z=0}
\left(V_{12}(\partial_x)\cdot \frac{e^{\langle a,x+z E\rangle
}}{\prod_{\phi\in \Phi\setminus\Phi_0}\langle \phi,x+zE\rangle
}\right)_{x=0}.
$$\\
\item Suppose $\Phi$ is unimodular.  Let $k_{12}=k(\Phi_0,\c_{12})$ be the
polynomial function on $W$ associated to the chamber $\c_{12}$ of
$\Phi_0$.    Let $K_{12}$ be any polynomial function on $\R^r$ extending $k_{12}$.
 Then,  if
$\langle E,\c_1\rangle > 0$, we have for $a\in \R^r$
$$
k(\Phi,\c_1)(a)-k(\Phi,\c_2)(a)=\Res_{z=0}
\left(K_{12}(\partial_x)\cdot \frac{e^{\langle a,x+z E\rangle
}}{\prod_{\phi\in \Phi\setminus \Phi_0}(1-e^{-\langle \phi,x+z
E\rangle })}\right)_{x=0}.
$$
\end{itemize}
\end{theorem}

In fact we will give a general version of the second part of this
theorem in Section \ref{sec:jumpgeneral}, where $\Phi$ is not
necessarily unimodular.

It is immediate to see that both formulae for the jumps $k(\Phi,
\c_1)-k(\Phi,\c_2)$ and $v(\Phi, \c_1)-v(\Phi,\c_2)$ are related by
the application of a generalized Khovanski-Pukhlikov differential
operator \cite{dm}, \cite{KP92}, \cite {BriVer97}.

We also demonstrate in various examples how to use these formulae
to compute the functions  $v(\Phi,\c)$ and $k(\Phi,\c)$.

\subsection*{List of Notations}
\[\begin{array}{ll}
U &r\mbox{-dimensional real vector space; $x \in U$; $dx$ Lebesgue measure on $U$.}\\
V &\mbox{dual of $U$; $a \in V$; $da$ dual Lebesgue measure on $V$.}\\
\Gamma &\mbox{a lattice in $V$; $\gamma\in \Gamma$.}\\
\Gamma^* &\mbox{dual  lattice in $U$; $  \ll \Gamma,\Gamma^*\rr\subset \Z$.}\\
T(\Gamma),T &\mbox{torus $U/\Gamma^*$.}\\
\Phi &\mbox{a sequence of nonzero  vectors in $\Gamma$ all on one side of a  hyperplane
;$\phi\in \Phi$.}\\
C(\Phi) &\mbox{cone generated by $\Phi$.}\\
<\Phi> &\mbox{vector space generated by $\Phi$.}\\
\Z\Phi  &\mbox{lattice  generated by $\Phi$ in $<\Phi>$.}\\
P(\Phi,a) &\mbox{convex polytope associated to $\Phi$ and $a$.}\\
\vol(\Phi,dx)(a) &\mbox{volume of $P(\Phi,a)$ for the quotient measure $dt/da$.} \\
k(\Phi)(a) &\mbox{number of integral points in $P(\Phi,a)$.}
\end{array}\]  \[\begin{array}{ll}
\c &\mbox{a chamber of $C(\Phi)$.}\\
v(\Phi,dx,\c) &\mbox{polynomial function coinciding with $\vol(\Phi,dx)$ on the cone $\c$.}\\
k(\Phi,\c) &\mbox{quasi-polynomial function coinciding with $k(\Phi)$ on $\c\cap \Gamma$.}\\
K &\mbox{quasi-polynomial function of $a$.}\\
\tau(\gamma) &\mbox{translation operator  $(\tau(\gamma)K)(a)=K(a-\gamma)$.}\\
D(\phi) &\mbox{difference operator  $(D(\phi)K)(a)=K(a)-K(a-\gamma)$.}\\
W &\mbox{hyperplane in $V$ defined for a fixed $E\in U$ by } \{a \in V | \langle a,E\rangle =0\}.\\
\Psi &\mbox{a sequence of   vectors in $\Gamma$  not in $W$; $\psi\in \Psi$.}\\
\Delta^+&\mbox{a sequence of   vectors all on one side of $W$; $\alpha\in \Delta^+$.}\\
\end{array}\]

\section{Partition functions}
\subsection{Definitions and notations}\label{sec:def}
Let $U$ be a  $r$-dimensional real vector space and $V$ be its dual
vector space. We  assume that $V$ is equipped with a lattice
$\Gamma$. We will usually denote by $x$ the variable in $U$ and by
$a$ the variable in $V$. We will see an element $P$ of $S(V)$ both
as a polynomial function on $U$ and a differential operator on $V$
via the relation $P(\partial_a)e^{\langle a,x\rangle }=P(x)e^{\langle a,x\rangle }$.

Let $\Phi=[\phi_1,\phi_2,\ldots,\phi_N]$ be a  sequence  of
non-zero, not necessarily distinct, linear forms on $U$ lying in an
open half space.  Assume that all the $\phi_k\in \Phi$ belong to the lattice
$\Gamma$. We denote by $<\Phi>$ the linear span of $\Phi$. Then
$\Phi$ generates a lattice in $<\Phi>$. We denote this lattice by
$\Z\Phi\subset \Gamma$.

We consider $\R^N$  with  basis $(\omega_1,\ldots,\omega_N)$ and let $A$ be
the linear map from $\R^N$ to the vector space $<\Phi>$ defined by
$A(\omega_k)=\phi_k,\  1\leq k\leq N$. The vectors $\phi_k$ are the
column vectors of the matrix $A$, and the map $A$ is surjective onto
$<\Phi>$. For $a\in <\Phi>$, we consider the convex polytope
$$P(\Phi,a):=\{t=(t_1,t_2, \ldots, t_N)\in  \R_{\geq 0}^N |\,  At=a\}.$$
In other words,
$$P(\Phi,a)=\{t=(t_1,t_2, \ldots, t_N)\in  \R_{\geq 0}^N | \ \sum_i t_i\phi_i=a\}.$$

Any polytope can be realized as a polytope $P(\Phi,a)$.

Let $C(\Phi)\subset <\Phi>$ be the cone generated by $\{\phi_1,\ldots,\phi_N\}$.
The cone $C(\Phi)$ is a pointed polyhedral cone. The dual cone
$C(\Phi)^*$ of $C(\Phi)$ is defined by  $C(\Phi)^*=\{ x\in U  \ |
\ll \phi, x \rr \geq 0 \  \hbox{for all} \ \phi \in \Phi \}$  and
its interior is non--empty. The polytope $P(\Phi,a)$ is empty if $a$
is not in  $C(\Phi)$. If $a\in <\Phi>$ is in the  relative interior of the cone
$C(\Phi)$, then the polytope $P(\Phi,a)$ has dimension $d:=N-\dim(<\Phi>)$.

We choose $dx$ on $<\Phi>^*$ and denote by $da$ the dual measure on
$<\Phi>$. Let  $dt$ be the Lebesgue measure on  $\R^N.$ The vector
space $\Ker(A)=A^{-1}(0)$ is of dimension $d=N-\dim(<\Phi>)$ and it
is equipped with the quotient Lebesgue measure $dt/da$ satisfying
$(dt/da)\wedge da=dt$.  For $a\in <\Phi>,$ $A^{-1}(a)$ is an affine
space parallel to $\Ker(A)$, thus also equipped with the Lebesgue
measure $dt/da$. Volumes of subsets of $A^{-1}(a)$ are computed with
this measure.  In particular we can define  for any $a\in <\Phi>$,
the number $\vol(\Phi)(a,dx)$ as being the volume of the convex set
$P(\Phi,a)$ in the affine space $A^{-1}(a)$ equipped with the
measure $dt/da$. If $dx$ is rescaled by $c>0$, then
$\vol(\Phi)(a,cdx)= c\vol(\Phi)(a,dx)$. By definition, if the
dimension of $P(\Phi,a)$ is less than $d$, $\vol(\Phi)(a,dx)$ is
equal to $0$.

\begin{definition} \label{defi.voluZ}
Let $<\Phi>$ be the subspace of $V$ generated by $\Phi$.
\begin{itemize}
\item  If $a\in <\Phi>$, define  $\vol(\Phi,dx)(a) ={\rm volume}(P(\Phi,a),dt/da).$
\item If $a\in <\Phi>$, define $k(\Phi)(a)=\vert P(\Phi,a)\cap \Z^N \vert.$
\end{itemize}

We extend the definition of the functions $\vol(\Phi,dx)(a)$ and
$k(\Phi)(a)$ as functions on $V$ by defining $\vol(\Phi,dx)(a)=0$ if
$a\notin <\Phi>$, $k(\Phi)(a)=0$ if $a\notin <\Phi>$.
\end{definition}

Clearly, $\vol(\Phi,dx)(a)=0$  if $a$ is not in $C(\Phi)$ and  $k(\Phi,a)=0$ if $a$ is not in $\Z\Phi\cap C(\Phi)$.

In the rest of this article, we will formulate many of our statements when $\Phi$ generates $V$,
as we can always reduce to this case replacing eventually $V$ by $<\Phi>$.

If $\Phi=[\phi_1,\phi_2,\ldots, \phi_r]$ consists of linearly
independent vectors, then the set $P(\Phi,a)$ is just one point
when $a\in C(\Phi)$ and is  empty when $a$ is not in the closed
cone $C(\Phi)$. Thus the function $\vol(\Phi)(a,dx)$ is just  the
characteristic function of the closed cone $C(\Phi)$  multiplied
by $|\det(\Phi)|^{-1}$ where the determinant is computed with
respect to the Lebesque measure $da$.
 Similarly, the function
$k(\Phi)(a)$ is the characteristic function of $C(\Phi)\cap
\sum_{i=1}^r\Z \phi_i$.

\begin{lemma}\label{lem:conV}
Assume $\Phi=\Phi'\cup \{\phi\}$ where $\Phi'$ generates $<\Phi>$. Then
 $$\vol(\Phi,dx)(a)=\int_{t\geq 0} \vol(\Phi',dx)(a-t\phi)dt$$
for any $a\in V$.
\end{lemma}
\begin{proof}
Indeed, decompose $\Phi=[\phi,\Phi']$. Then
$$P(\Phi,a)=\{[t,t']; t\geq 0,  t'\in P(\Phi',a-t \phi)\}.$$
The proof follows by Fubini.
\end{proof}\bigskip\bigskip

By induction, we obtain the following corollary.
\begin{corollary}\label{co:continuous}
If $\Phi$ consists of linearly independent vectors, the function
$\vol(\Phi,dx)(a)$ is continuous on $C(\Phi)$.

If $\Phi$ generates $V$ with $|\Phi|>\dim V$, then the function
$\vol(\Phi,dx)(a)$ is continuous on $V$.
\end{corollary}

For an element $\gamma$ in $V$, define the translation operator
$\tau(\gamma)$ on functions $k(a)$ on $V$ by the formula: if $a\in
V$, then
$$(\tau(\gamma) k)(a)=k(a-\gamma).$$
The difference
operator $D(\gamma)=1-\tau(\gamma)$  acts on functions $k(a)$ on $V$ by the
formula:
$$(D(\gamma) k)(a)=k(a)-k(a-\gamma).$$

The following lemma is obvious from the definition.
\begin{lemma}
Let $\phi\in \Phi$ and $a\in \Gamma$. Then
$$k(\Phi)(a)=\sum_{n=0}^{\infty} k(\Phi\setminus \{\phi\})(a-n\phi).$$
\end{lemma}

The following relation follows immediately.

\begin{lemma}\label{diffk}
Let $\phi\in \Phi$ and $a\in \Gamma$. Then
$$(D(\phi)k(\Phi))(a)=k(\Phi\setminus \{\phi\})(a).$$
In particular, $(D(\phi)k(\Phi))(a)$ is equal to $0$ if $a$ is not
in the subspace of $V$ generated by $\Phi\setminus \{\phi\}$.
\end{lemma}

\begin{lemma}\label{face}
Assume $\Phi$ generates $V$. Let $W$ be a hyperplane in $V$ such
that $W\cap C(\Phi)$ is a facet of $C(\Phi)$. Let $\Phi_0$ be the
sequence $\Phi\cap W$ which spans $W$. If $a\in W$, then
$k(\Phi)(a)=k(\Phi_0)(a)$.
\end{lemma}
\begin{proof}
As $W\cap C(\Phi)$ is a facet of $C(\Phi)$, if  $a\in W\cap
C(\Phi)$, any solution of $a=\sum_{i=1}^N y_i\phi_i$ with $y_i\geq
0$ will have $y_i=0$ for $\phi_i\notin W$.
\end{proof}\bigskip

The following lemma is also obtained immediately from Fubini's
theorem applied to the integral $\int_{\R_{\geq 0}^N}e^{-\langle
\sum_{i=1}^N t_i\phi_i,x\rangle }dt_1dt_2\cdots dt_N$ decomposed
along the fibers of the map $A:\R_{\geq 0}^N\to C(\Phi)$, or to the
analogous discrete sum.

\begin{lemma}\label{laplaceV}
For $x$ in the interior of $C(\Phi)^*$,
$$ \int_{C(\Phi)} \vol(\Phi,dx)(a) e^{-\langle a,x\rangle } da =\frac{1}{\prod_{\phi \in \Phi} \langle \phi,x\rangle },$$
$$ \sum_{a\in  C(\Phi) \cap \Gamma} k(\Phi)(a) e^{-\langle a,x\rangle }=
\frac{1}{\prod_{\phi \in \Phi}1-e^{-\langle \phi,x\rangle }}.$$
\end{lemma}

\subsection{Chambers and the qualitative behavior of partition functions}
In this section, we assume that $\Phi$ generates $V$. For any  subset
$\nu$  of   $\Phi$,  we denote by $C(\nu)$ the closed cone generated
by $\nu$. We denote by $C(\Phi)_{\rm sing}$ the union of the cones
$C(\nu)$  where $\nu$ is any subset of $\Phi$ of cardinality
strictly less than $r=\dim(V)$. By definition, the set
$C(\Phi)_{{\rm reg}}$ of $\Phi$-regular elements is the complement
of  $C(\Phi)_{\rm sing}$. A connected component of $C(\Phi)_{\rm
reg}$ is called a  chamber.  We remark that,  according to our
definition, the exterior of $C(\Phi)$ is itself a chamber denoted by
$\c_{{\rm ext}}$. The chambers contained in $C(\Phi)$ will be called
interior chambers. If  $\c$ is  a chamber, and $\sigma$ is a basis
of $V$ contained in $\Phi$, then either $\c\subset C(\sigma)$, or
$\c\cap C(\sigma)=\emptyset$, as the boundary of $C(\sigma)$ does
not intersect $\c$.

Let $\Phi'\subset \Phi$ be such that  $\Phi'$ generates $V$.  If $\c$ is a
chamber for $\Phi$, there exists a unique chamber $\c'$ for $\Phi'$
such that $\c\subset \c'$.

A wall of $\Phi$ is a (real) hyperplane generated by $r-1$ linearly
independent elements of $\Phi$. It is clear that the boundary of a
chamber $\c$ is contained in  an union  of walls.

We now define the notion of a quasi-polynomial function on the lattice $\Gamma$.
Let  $\Gamma^*$ be the dual lattice
of $\Gamma$. An element
$x\in U$ gives rise to the exponential function $e_x(a)=e^{2i\pi
\langle x,a\rangle }$ on $\Gamma$.
 Remark that the function
$e_x(a)$ depends only of the class of $x$ (still denoted by $x$) in the torus
$T(\Gamma):=U/\Gamma^*$.

Let $M$ be a positive integer. A quasi-polynomial function with
period $M$ on $\Gamma$ is a function $K$ on $\Gamma$ of the form
$K(a)=\sum_{x\in F}e_x(a) P_x(a)$ where $F$ is a finite set of
points of $U$ such that $MF\subset \Gamma^*$ and $P_x$ are
polynomial functions on $V$. Then the restriction of the function
$K$ to cosets $h+M\Gamma$ of $\Gamma/M\Gamma$ coincide with the
restriction to $h+M\Gamma$ of a polynomial function on $V$. If the
degree of the polynomial $P_x(a)$ is less or equal to $k$ for all
$x\in F$, we say that $K$ is a quasi-polynomial function of degree
$k$ and period $M$.

If $\Gamma=\Z$ and  $ \gamma\in \C^* $ is a $M^{th}$ root of unity,
the function $n\mapsto n^k \gamma ^n$ is a quasi-polynomial function
on $\Z$ of period $M$ and degree $k$.

If $C$ is an affine  closed cone in $V$ with non empty interior, a
quasi-polynomial function on $\Gamma$ vanishing on $\Gamma\cap C$ is identically equal to $0$ on $\Gamma$.

If $\gamma \in \Gamma$, the difference operator $D(\gamma)k(a)=
k(a)-k(a-\gamma)$ leaves the space of  quasi-polynomial functions on
$\Gamma$ stable.

The following theorem is well known (see
\cite{dm},\cite{BriVer97},\cite{Sze-ve}, \cite{dp}). See  yet
another  proof in the forthcoming article \cite{ve08}.
\begin{proposition}\label{qualitative}
Let $\c$ be an interior chamber of $C(\Phi)$.
\begin{itemize}
\item
There exists a unique homogeneous polynomial function $v(\Phi,dx,\c)$ of degree $d$ on $V$
such that, for $ a\in \overline{\c}$,
$$\vol(\Phi,dx)(a)=v(\Phi,dx,\c)(a).$$
\item
There exists a unique quasi-polynomial function $k(\Phi,\c)$ on
$\Gamma$ such that,  for $ a\in \overline{\c}\cap \Gamma$,
$$k(\Phi)(a)=k(\Phi,\c)(a).$$
\end{itemize}
\end{proposition}

\begin{remark}\label{uni}
The sequence $\Phi$ is called unimodular if, for any subset $\sigma$
of $\Phi$ forming a basis of $V$,  the subset $\sigma$ is a basis of
$\Z\Phi$. In other words, we have $|\det(\sigma)|=1$,  where the
determinant is computed using the volume $da$ giving volume $1$ to
 a fundamental domain for $\Z\Phi$. In this particular
case, the function $k(\Phi,\c)$ is polynomial on $\Z\Phi$.
\end{remark}

In the next lemma, we list differential equations satisfied by the
polynomial function $v(\Phi,dx,\c)$.

\begin{lemma}\label{lem:diffvc}
Let $\phi\in \Phi$. If $\Phi\setminus \{\phi\}$ does not generate
$V$, then $\partial(\phi)v(\Phi,dx,\c)=0$.

If $\Phi\setminus\{\phi\}$ generates $V$, let $\c'$ be the chamber
of $\Phi\setminus\{\phi\}$ containing $\c$, then $\partial(\phi)
v(\Phi,dx,\c)=v(\Phi\setminus \{\phi\},dx,\c')$.
\end{lemma}
\begin{proof}
If $\Phi_0=\Phi\setminus \{\phi\}$ is contained in a wall $W$,
then $V=W\oplus \R \phi$, and it is immediate to see that an
interior chamber $\c$ for $\Phi$ is of the form
$\c=\c_0+\R_{>0}\phi$, where $\c_0$ is a chamber for $\Phi_0$.
If
$a=w+t\phi$ with $w\in W$ and $t>0$, then
$\vol(\Phi,dx,\c)(w+t\phi)=\vol(\Phi_0,dx_0,\c_0)(w)$, with
$dx_0d\phi=dx$. This proves the first statement.

To prove the second statement, if $a\in \c$,  we use the following
relation (as given in Lemma \ref{lem:conV})
$$\vol(\Phi,dx)(a)-\vol(\Phi,dx)(a-\epsilon \phi)=\int_{t= 0}^{\epsilon} \vol(\Phi_0,dx)(a-t\phi)dt.$$
\end{proof}\bigskip

\begin{corollary}\label{cor:dah}
Let $\Phi_0\subset \Phi$ such that $\Phi_0$ does not generate $V$.
Then $$(\prod_{\phi\in
\Phi\setminus\Phi_0}\partial(\phi))v(\Phi,dx,\c)=0.$$
\end{corollary}

In the next lemma, we list difference equations   satisfied by the
quasi-polynomial function $k(\Phi,\c)$.

\begin{lemma}\label{lem:diffkc}
Let $\phi\in \Phi$. If $\Phi\setminus \{\phi\}$ does not generate
$V$, then $D(\phi)k(\Phi,\c)=0$.

If $\Phi\setminus\{\phi\}$ generates $V$, let $\c'$ be the chamber
of $\Phi\setminus\{\phi\}$ containing $\c$, then $D(\phi)
k(\Phi,\c)=k(\Phi\setminus\{\phi\},\c')$
\end{lemma}
\begin{proof}
By Lemma \ref{diffk}, the function $k(\Phi)$ satisfies $D(\phi)
k(\Phi)=k(\Phi \setminus \{\phi\})$. Considering this relation on an
affine subcone $S$ of $\overline{\c}$ such that $S-\phi$ does not
touch the boundary of $\c$, we obtain the relations of the lemma.
\end{proof}

\section{Two polynomial functions}

\subsection{Residue formula }
Let $\CL$ be the space of Laurent series in one variable $z$:
$$\CL:=\{f(z)=\sum_{k\geq k_0} f_k z^k\}.$$
For $f\in \CL$, we denote by $\Res_{z=0}f(z)$ the coefficient $f_{-1}$ of $z^{-1}$.
If $g$ is a germ of meromorphic function at $z=0$,  then  $g$  gives rise to an element
 of $\CL$ by considering the Laurent series  at $z=0$ and
we still denote by  $\Res_{z=0}g$ its residue at $z=0$.
If $g=\frac{d}{dz}f$, then $\Res_{z=0}g=0$.
\medskip

With the notation of Section \ref{sec:def}, let $E$ be a  vector in
$U$. It defines a hyperplane $W=\{a \in V | \langle a,E\rangle =0\}$
in $V$.
\begin{definition}\label{defPol}
Let $P$ be a polynomial function on  $V$ and let $\Psi$ be a
sequence of vectors not belonging to $W$. We define, for $a\in V$,
\begin{itemize}
\item
${\rm Pol}(P,\Psi,E)(a)=\Res_{z=0} \left(P(\partial_x)\cdot
\frac{e^{\langle a,x+z E\rangle }}{\prod_{\psi\in \Psi}\langle
\psi,x+zE\rangle }\right)_{x=0}.$\\

\item
${\rm Par}(P,\Psi,E)(a)=\Res_{z=0} \left(P(\partial_x)\cdot
\frac{e^{\langle a,x+z E\rangle }}{\prod_{\psi\in
\Psi}(1-e^{-\langle \psi,x+zE\rangle}) }\right)_{x=0}.$
\end{itemize}
\end{definition}

It is easy to see  that  ${\rm Pol}(P,\Psi,E)(a)$ as well as ${\rm
Par}(P,\Psi,E)(a)$  are  polynomial functions of $a\in V$.

\begin{lemma}\label{lem:Pp}
The functions ${\rm Pol}(P,\Psi,E)$ and ${\rm Par}(P,\Psi,E)$ depend
only on the restriction $p$ of $P$ to $W$.
\end{lemma}

\begin{proof}
If $p=0$, then $P=EQ$ where $Q$ is a  polynomial function on $V$.
Then
\[\begin{array}{ll}
P(\partial_x) F(x+zE)&=\frac{d}{d\epsilon}Q(\partial_x) F(\epsilon E+x+z E)_{\epsilon=0}\\
&=\frac{d}{d\epsilon}Q(\partial_x) F(x+(z+\epsilon) E)_{\epsilon=0}\\
&=\frac{d}{dz} (Q(\partial_x) F(x+z E))\\
\end{array}\]
so that the residue $\Res_{z=0}$ vanishes on the function
$z\mapsto P(\partial_x) F(x+zE)_{x=0}.$ \end{proof}\bigskip

We can then give the following definitions.
\begin{definition}
Let $p$ be a polynomial function on $W$.
We define $${\rm Pol}(p,\Psi,E):={\rm Pol}(P,\Psi,E)$$
where $P$ is any polynomial on $V$ extending $p$.

We define $${\rm Par}(p,\Psi,E):={\rm Par}(P,\Psi,E)$$
where $P$ is any polynomial on $V$ extending $p$.
 \end{definition}

In the following, given  polynomials $p,q,\ldots$ on $W$, we denote
by $P,Q,\ldots$ polynomials on $V$ extending $p,q,\ldots$.

\subsection{Some properties}
Let us list some properties satisfied by the function ${\rm
Pol}(p,\Psi,E)$.

We first remark that if we replace $\psi$ in $\Psi$ by $c\psi$ with
$c\neq 0$,  then ${\rm Pol}(p,\Psi,E)$ becomes $\frac{1}{c}{\rm
Pol}(p,\Psi,E)$.

We now discuss how ${\rm Pol}(p,\Psi,E)$ transforms under the action
of differentiation.
\begin{proposition}\label{diffPv}
Let $\psi\in \Psi$. Then
$$\partial(\psi) {\rm Pol}(p, \Psi,E)={\rm Pol}(p, \Psi\setminus \{\psi\},E).$$

Let $w\in W$. Then
$$\partial(w) {\rm Pol}(p,\Psi,E)={\rm Pol}(\partial(w)p,\Psi,E).$$

\end{proposition}
\begin{proof}
The first formula follows immediately  from the definition.

For the second part of the proposition, we will use the following
lemma, which is implied by the relation
$P(\partial_x)\langle x,w\rangle -\langle x,w\rangle P(\partial_x)=(\partial(w)P)(\partial_x)$.

\begin{lemma}\label{Pdiff}
For any function $J(x)$ of $x\in U$,
$$\Big(P(\partial_x)\langle x,w\rangle  J(x)\Big)_{x=0}=\Big((\partial(w)P)(\partial_x) J(x)\Big)_{x=0}.$$
\end{lemma}
\bigskip

Now, as $\langle w,E\rangle =0$, for $J(x,z)=\frac{1}{\prod_{\psi\in
\Psi}\langle \psi,x+zE\rangle }$, we have
\[\begin{array}{ll}
\partial(w) \Res_{z=0}\left(P(\partial_x)
e^{\langle a,x+z E\rangle} J(x,z)\right)_{x=0}&=
\Res_{z=0}\left(P(\partial_x)\langle w,x\rangle  e^{\langle a,x+z E\rangle}
J(x,z)\right)_{x=0}\\
&=\Res_{z=0}\left((\partial(w)P)(\partial_x) e^{\langle a,x+z
E\rangle} J(x,z)\right) _{x=0}. \end{array}\] So we obtain the
formula of the proposition.
\end{proof}\bigskip

\begin{lemma}\label{restrictionV}
\begin{itemize}
\item If $\Psi=\{\psi\}$, then for $w\in W$ and $t\in \R$,
 $$ {\rm Pol}(p, \{\psi\},E)(w+t\psi)={\rm Par}(p,  \{\psi\},E)(w+t\psi)=\frac{p(w)}{\ll \psi, E\rr}.$$
 \item If
$|\Psi|>1$, then the restriction of ${\rm Pol}(p,\Psi,E)$ to
$W$ vanishes of order $|\Psi|-1$.
\end{itemize}
\end{lemma}
\begin{proof}
Let $U_0=\{x| \ll \psi,x\rr =0\}$. We write $ U=U_0\oplus \R E$. The space $S(U_0)$ is isomorphic to the space of polynomial functions on $W$. We may choose $P$ in $S(U_0)$.
We write $x=x_0+x_1 E$, with $x_0\in U_0$.  In these coordinates
 $\ll\psi,x+zE\rr=(x_1+z) \ll \psi,E\rr$ is independent of $x_0$.
  So we can set  $x_1=0$ in the formula
   $$\Res_{z=0} \left(P(\partial_{x_0})\cdot
\frac{e^{\langle a,x+z E\rangle }}{(x_1+z) \ll\psi,E\rr}\right)_{x=0}$$
 and  the residue is computed for a function that have a simple pole at $z=0$.
The formula follows. The other points are also easy to prove.
\end{proof}\bigskip

Let us list some difference equations satisfied by the function
${\rm Par}(p,\Psi,E)$.

\begin{proposition}\label{diffpar}
Let $\psi\in \Psi$.
 Then
$$D(\psi) {\rm Par}(p, \Psi,E)={\rm Par}(p, \Psi\setminus \{\psi\},E).$$

Let $w\in  W$. Then
$$\tau(w) {\rm Par}(p,\Psi,E)={\rm Par}(\tau(w)p,\Psi,E).$$
\end{proposition}

\begin{proof}
The first formula follows immediately from the  definition.

The translation operator $\tau(w)$  satisfies the relation
$$P(\partial_x)e^{-\ll w,x\rr}=e^{-\ll w,x\rr}(\tau(w)P)(\partial_x).$$
Thus, the  second formula follows from the same argument as in the
proof of the second item in Proposition \ref{diffPv}.
\end{proof}\bigskip

\section{Wall crossing formula for the volume}

In this  section, we give two formulae for the jump of the volume
function across a wall. The first one uses convolutions of Heaviside
distributions and is in the spirit of Paradan's formula (\cite{Par},
Theorem 5.2) for the jump of the partition function. The second one
is a one dimensional residue formula.

\subsection{Inversion formula}
We will need  some formulae for Laplace transforms in dimension $1$.
For $z>0$ and   $k\geq 0$ an integer, we have
\begin{equation}\label{laplacecontinuous}
\frac{1}{z^{k+1}}=\int_{0}^{\infty} \frac{t^{k}}{k!} e^{-tz} dt.
\end{equation}

Consider the Laplace transform
$$L(p)(z)=\int_{\R^+} e^{-tz}p(t)dt.$$
Assume that $p(t)=\sum_{ik}c_{ik}p_{i,k}(t)$   is a linear combination
of  the functions  $p_{i,k}(t)=e^{-t x_i}\frac{t^k}{k!}$. We assume  that  $x_i>0$.  Then, the integral defining $L(p)$ is convergent. We have
\begin{equation}\label{lap}
L(p)(z)=\sum_{i,k}\frac{c_{ik}}{(z+x_i)^{k+1}}.
\end{equation}

The following inversion formula is immediate to verify.
\begin{lemma} Let $R>0$. Assume that $|x_i|<R$ for all $i$.
Then we have
\begin{equation}\label{inversionexppol}
p(t)=\frac{1}{2i\pi}\int_{|z|=R} L(p)(z)e^{tz}dz.
\end{equation}
\end{lemma}
Reciprocally, if $p$ is a continuous function on $\R$ such that $L(p)(z)$ is convergent and given by
Formula (\ref{lap}), then $p$ is given by Equation (\ref{inversionexppol}).

If $p(t)=\sum_k c_k\frac{t^k}{k!}$  is a polynomial (that is all the
elements $x_i$ are equal to $0$), then  $L(p)(z)$ is  the Laurent
polynomial $\sum_{k}c_k z^{-k-1}$, and the inversion formula above
reads
\begin{equation}\label{inversionpol}
p(t)=\Res_{z=0}L(p)(z) e^{tz}.
\end{equation}

\subsection{Convolution of measures}
Let $E\in U$ be a non zero linear form on $V$  and $W\subset V$
the corresponding hyperplane.
 Let $V^+=\{a \in V |
\langle a,E\rangle >0\}$ and $V^-=\{a \in V | \langle a,E\rangle <0\}$ denote the corresponding
open half spaces. Let $\Delta^+=[\alpha_1,\alpha_2,\ldots,
\alpha_Q]$ be a sequence of vectors contained in $V^+$.  Consider
the span $<\Delta^+>$ of $\Delta^+$. We choose a Lebesgue measure
$da$ on $<\Delta^+>$ with dual measure $dx$ on $<\Delta^+>^*$. We
define the continuous function $v(\Delta^+,dx)(a)$ on the cone
$C(\Delta^+) \subset <\Delta^+>$ such that, for $x\in C(\Delta^+)^*$, we have

\begin{equation}\label{defv}
\frac{1}{\prod_{\alpha\in \Delta^+}\langle \alpha,x\rangle }=\int_{C(\Delta^+)}
v(\Delta^+,dx)(a)e^{-\langle a,x\rangle }da.
\end{equation}

By Lemma~\ref{laplaceV}, $v(\Delta^+,dx)(a)=\vol(\Delta^+,dx)(a)$.

We choose the Lebesgue measure $dw=da/dt$ on $W\cap <\Delta^+>$ where $t=\langle a,E\rangle $.
The measure $dw$ determines a measure on all affine spaces $W\cap (a+ <\Delta^+>)$.

\begin{proposition}
Let $p$ be a polynomial function on $W$.  We define
$$(p*v(\Delta^+,dx))(a)=\int_{W\cap (a+<\Delta^+>)} p(w) v(\Delta^+,dx)(a-w) dw.$$
Then, for $a\in V^+$, we have
$$(p*v(\Delta^+,dx))(a) = {\rm Pol}(p,\Delta^+,E)(a) \,\, {\rm for} \, \, a\in V^+.$$
\end{proposition}

Remark that $p*v(\Delta^+,dx)$  depends only on the
choice of $E$.  Indeed,  $p*v(\Delta^+,dx)$ is
the convolution of two functions, one of which depends on the
measure $dx$, while the convolution  depends on the measure $dw$.  We see
that finally this depends only of the choice of $E$.

We also remark that, for fixed $a\in V^+$, the integral defining
$p*v(\Delta^+,dx)$ is in fact over the compact set $W\cap
(a-C(\Delta^+))$ where $v(\Delta^+,dx)(a-w)$ is not equal to zero .

\begin{proof}
We decompose $W=W_0\oplus W_1$, where $W_0=W\cap <\Delta^+>$. Then,
we can write $a\in V$ as $a=tF+w_0+w_1$, with $\langle F,E\rangle
=1$. If $p(w_0+w_1)=p_0(w_0) p_1(w_1)$,  we see that
$(p*v(\Delta^+,dx))(tF+w_0+w_1)=p_1(w_1)
(p_0*v(\Delta^+,dx))(tF+w_0)$. Hence, it is sufficient to prove the
proposition in the case where $\Delta^+$ generates $V$. Then
$$(p*v(\Delta^+,dx))(a)=\int_W p(w)v(\Delta^+,dx)(a-w)dw.$$
The polynomial nature of $(p*v(\Delta^+,dx))(a)$ is clear intuitively.
In any case, we will prove the explicit formula of the proposition,
which gives a polynomial formula for $(p*v(\Delta^+,dx))(a)$.

We need to compute, for $a\in V^+$, $I(a):=\int_W p(w)
v(\Delta^+,dx)(a-w) dw$. This integral is over a compact subset of
$W$. Let  $P\in S(U)$ be  a polynomial function on $V$ extending $p$. We may
write
$$I(a)=\left( P(\partial_x)\cdot \int_{W}v(\Delta^+,dx)(a-w) e^{\langle w,x\rangle }dw \right)_{x=0}.$$
Define $$g_x(a)=\int_{W}v(\Delta^+,dx)(a-w) e^{-\langle a-w,x\rangle
}dw.$$ Then $g_x(a)$ depends analytically on the variable $x\in U$,
and we have
\begin{equation}\label{myI}
I(a)=\left( P(\partial_x)\cdot e^{\langle a,x\rangle
}g_x(a)\right)|_{x=0}.
\end{equation}
The function $a\mapsto g_x(a)=\int_{W}v(\Delta^+,dx)(a-w)
e^{-\langle a-w,x\rangle }dw$ is a continuous function of $a$ modulo
$W$, that is, it is a continuous function of  the variable
$t=\langle a,E\rangle \geq 0$ when $a\in V^+$. We then write
$g_x(t)=g_x(tF)=\int_{W}v(\Delta^+,dx)(tF-w) e^{-\langle
tF-w,x\rangle }dw.$

To identify the function $g_x(t)$, we compute its Laplace transform
in  one variable. Let $z>0$. If $x$ is in $C(\Delta^+)^*$, the
integral defining $L(g_x)$ is convergent  and we have
\[\begin{array}{ll}
L(g_x)(z)=\int_{t>0}e^{-tz}g_x(tF)dt&=\int_{V^+} e^{-\langle
a,zE\rangle}v(\Delta^+,dx)(a)e^{-\langle a,x\rangle }da\\
&=\frac{1}{\prod_{\alpha\in \Delta^+}\langle \alpha, x+zE\rangle }
\end{array}\]
by  Formula (\ref{defv}). Here $ \langle \alpha, x+zE\rangle
=d_\alpha z+\langle \alpha, x\rangle$ with $d_\alpha=\langle \alpha,
E\rangle>0$ and $\langle \alpha, x\rangle >0$.

Thus, by partial fraction decomposition, $L(g_x)(z)$ is a function
of the type given by Formula (\ref{lap}). By the inversion formula
for the Laplace transform in one variable, we obtain that (for $x$
small enough)
$$g_x(a)=\frac{1}{2i\pi} \int_{|z|=1}\frac{e^{\langle a,z E\rangle }}
{\prod_{\alpha\in \Delta^+} \langle \alpha,x+zE\rangle} dz.$$

Thus Formula (\ref{myI}) becomes
$$I(a)= P(\partial_x)\cdot\left(e^{\langle a,x\rangle}\frac{1}{2i\pi} \int_{|z|=1}
\frac{e^{\langle a,z E\rangle }}{\prod_{\alpha\in \Delta^+} \langle
\alpha, x+zE\rangle} dz\right)_{x=0}$$
$$=\frac{1}{2i\pi} \int_{|z|=1}\left( P(\partial_x)\cdot\frac{e^{\langle a,x+zE\rangle }}
{\prod_{\alpha\in \Delta^+} \langle \alpha, x+zE\rangle}
\right)_{x=0} dz.$$

The function in the integrand has a Laurent series at $z=0$ with polynomial coefficients in $a$
of  the form $\sum_{k}g_k(a)z^k$. Thus we obtain
$$I(a)=\Res_{z=0} \left(P(\partial_x)\cdot \frac{e^{\langle a,x+zE\rangle }}
{\prod_{\alpha\in \Delta^+} \langle \alpha,x+zE\rangle}
\right)_{x=0}.$$ This shows  that   $I(a)$ coincide with the  polynomial
function ${\rm Pol}(p,\Delta^+,E)$ on
$V^+$.
\end{proof}\bigskip\medskip

We also remark that, if $p$ is homogeneous, ${\rm
Pol}(p,\Delta^+,E)$ is homogenous in $a$ of degree
$|\Delta^+|-1+\rm{deg}(p)$.

\subsection{The jump for the volume  function}\label{sec:jumpvol}
Let  $\vol(\Phi,dx)$  be the locally polynomial function on the
cone $C(\Phi)$ generated by $\Phi$.  Let $W$ be a wall, determined
by a  vector $E \in U$.  Let $V^+$ and $V^-$ denote the
corresponding open half spaces. Define $\Phi_0=\Phi\cap W$; this
is a sequence of vectors in $W$ spanning $W$.

Let $\c_1\subset V^+$ and  $\c_2\subset V^-$  be two chambers  on two sides of $W$ and adjacent. Here, we mean that
$ \overline{\c_1}\cap \overline{\c_2}$ has non empty relative interior in $W$. Thus
$ \overline{\c_1}\cap \overline{\c_2}$ is contained in the closure of a chamber $\c_{12}$ of
$\Phi_0$. We choose the measure $dw$ on $W$ such that $da=dw dt$
with $t=\langle a,E\rangle $. We write $$\Phi=[\Phi_0,\Phi^+,\Phi^-]$$ where
$\Phi^+=\Phi\cap V^+$ and $\Phi^-=\Phi\cap V^-$.

Let $$R_+(\Phi)=[\phi|\phi\in \Phi^+]\cup [-\phi|\phi\in \Phi^-].$$ By
construction, the sequence $R_+(\Phi)$ is contained in $V^+$.

\begin{theorem}\label{jumpvol}
Let $v_{12}=v(\Phi_0,dw,\c_{12})$ be the polynomial function on $W$
associated to the chamber $\c_{12}$ of $\Phi_0$.  Then, if $\ll E,\c_1\rr>0$,
\begin{equation}\label{Leq}
v(\Phi,dx,\c_1)-v(\Phi,dx,\c_2)={\rm Pol}(v_{12},\Phi\setminus
\Phi_0,E).
\end{equation}
\end{theorem}

\begin{remark}
We have $${\rm Pol}(v_{12},\Phi\setminus \Phi_0,E)=(-1)^{|\Phi^-|}
{\rm Pol}(v_{12},R_+(\Phi),E).$$ Thus, by results of the preceding
section, the difference of the volume functions
$v(\Phi,dx,\c_1)-v(\Phi,dx,\c_2)$ coincides on $V^+$, up to sign,
with the convolution of the {\bf polynomial measure} $v_{12}(w)dw$
associated to the chamber $\c_{12}$ and with  the Heaviside
distributions associated to the  vectors $ \psi\in R_+(\Phi)$. This
is in the line of Paradan's description of the jump formula for
partition functions (\cite{Par}, Theorem 5.2) .
\end{remark}

\begin{proof}
Denote by ${\rm Leq}(\Phi)$ the left hand side and by ${\rm
Req}(\Phi)$  the right hand side of Equation (\ref{Leq}) above.

We will first verify the claim in the theorem when  there is only
one vector $\phi$ of $\Phi$ that does not lie in $W$. We can suppose
that $\Phi^+=\{\phi\}$ and that $\ll E,\phi\rr=1$. Then the chamber
$\c_1$ is equal to $\c_{12}\times \R_{>0} \phi$, while $\c_2$ is the
exterior chamber. In this case, $v(\Phi,dx,\c_1)(w+t
\phi)=v(\Phi_0,(dw)^*,\c_{12})(w)=v_{12}(w)$, while
$v(\Phi,dx,\c_2)=0$. The equation (\ref{Leq}) follows from the first
item of Lemma \ref{restrictionV}.

If not, then $|\Phi|>\dim (V)$  and the function $\vol(\Phi,dx)$ is
continuous on $V$ by Corollary \ref{co:continuous}. Let $\phi$ be a
vector in $\Phi$ that does not lie in $W$. We may assume that
$\phi\in V^+$. Then the sequence $\Phi'=\Phi\setminus \{\phi\}$ will
still span $V$. The intersection of $\Phi'$ with $W$ is $\Phi_0$.
If $\c'_1$ and $\c'_2$ are the chambers of $\Phi'$ containing $\c_1$
and $\c_2$ respectively, they are adjacent with respect to $W$.  As
$\Phi'\cap W=\Phi_0$,  the polynomial $v'_{12}$ attached to
$\c_{12}$ and $\Phi'\cap W$ is equal to $v_{12}$. By Lemma
\ref{lem:diffvc}, we have
$$\partial(\phi) (v(\Phi,\c_1)-v(\Phi,\c_2))=v(\Phi',\c'_1)-v(\Phi',\c'_2).$$

By Proposition \ref{diffPv},
$$\partial(\phi) {\rm Pol}(v_{12}, \Phi\setminus\Phi_0,E)={\rm Pol}(v_{12},\Phi'\setminus\Phi_0,E).$$
By induction, we obtain $\partial(\phi) ({\rm Leq}(\Phi)-{\rm
Req}(\Phi))=0$. Thus the polynomial function  ${\rm Leq}(\Phi)-{\rm
Req}(\Phi)$ is constant in the direction $\R\phi$. The left hand
side vanishes on $W$, as the function $\vol(\Phi)$ is continuous on
$V$. The right hand side also vanishes by Lemma \ref{restrictionV}.
This establishes the theorem.
\end{proof}\bigskip\bigskip
%

\begin{figure}[!h]
\begin{center}\includegraphics{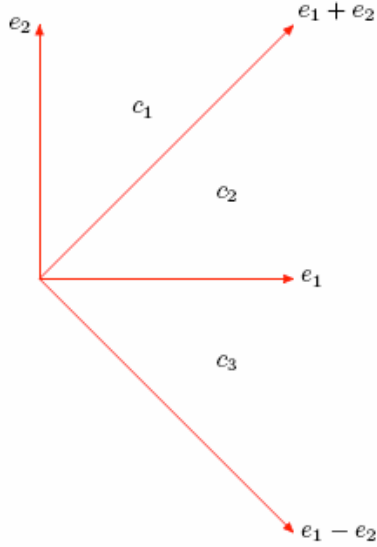}
\end{center}\caption{Chambers of $B_2$}\label{b2}
\end{figure}

Consider a vector space $V$ with basis $\{e_i,\; i:1,\ldots,r\}$; we
denote its dual basis by $\{e^i\}$.  The set
\[\Phi(B_r)=\{e_i, \; 1\leq i \leq r\} \cup \{e_i+e_j, \; 1\leq i<j \leq
r\}\cup \{e_i-e_j, \; 1\leq i<j \leq r\}\] is the set of positive
roots for the system of type $B_r$ and generates $V$.  We will
denote a vector $a\in V$ by $a=\sum_{i=1}^r a_ie_i$; it lies in
$C(\Phi(B_r))$ if and only if $a_1+\cdots+a_i\geq 0$ for all
$i:1,\ldots,r$.   This will be our notation for this root system in
subsequent examples.

\begin{example}\label{b2vol} We consider the root system of type $B_2$ (see
Figure~\ref{b2}) with $\Phi=\{e_1,e_2,e_1-e_2,e_1+e_2\}$. We will
calculate $v(\Phi,\c)$ for all the chambers using our formula in
Theorem \ref{jumpvol} iteratively starting from the exterior
chamber.

(i) Jump from the exterior chamber to $\c_1$ ($W=\R e_2$): In this
case $E=e^1$, $\Phi_0=\{e_2\}$, $\Phi^+=\{e_1+e_2,e_1,e_1-e_2\}$ and
$\Phi^-=\emptyset$.
\[\begin{array}{ll}
v(\Phi,\c_1)(a)-v(\Phi,\c_{\rm ext})(a)&={\rm Pol}(1,\Phi\setminus
\Phi_0,E)(a)\\
&=\Res_{z=0} \left(\frac{e^{\ll a,x+ze^1\rr }}{\prod_{\phi\in \Phi^+
\cup \Phi^-}\langle \phi,x+ze^1\rangle
}\right)_{x=0}\\
&=\Res_{z=0} \left(\frac{e^{a_1(x_1+z)+a_2x_2
}}{(x_1+z+x_2)(x_1+z)(x_1+z-x_2)}\right)_{x=0}\\
v(\Phi,\c_1)(a)&=\Res_{z=0}
\left(\frac{e^{a_1z}}{z^3}\right)=\frac{1}{2}a_1^2.
\end{array}\]

(ii) Jump from $\c_1$ to $\c_2$ ($W=\R (e_1+e_2)$): We have
$E=e^1-e^2$, $\Phi_0=\{e_1+e_2\}$, $\Phi^+=\{e_1,e_1-e_2\}$ and
$\Phi^-=\{e_2\}$.
\[\begin{array}{ll}
v(\Phi,\c_2)(a)-v(\Phi,\c_1)(a)&={\rm Pol}(1,\Phi\setminus
\Phi_0,E)(a)\\ &=\Res_{z=0} \left(\frac{e^{a_1(x_1+z)+a_2(x_2-z)
}}{(x_1+z)(x_1-x_2+2z)(x_2-z)}\right)_{x=0}\\
&=-\Res_{z=0}\left(\frac{e^{(a_1-a_2)z}}{2z^3}\right)=-\frac{1}{4}(a_1-a_2)^2.
\end{array}\]
Using (i),
$v(\Phi,\c_2)(a)=\frac{1}{2}a_1^2-\frac{1}{4}(a_1-a_2)^2=\frac{1}{4}(a_1+a_2)^2-\frac{1}{2}a_2^2$.

(iii) Jump from $\c_2$ to $\c_3$ ($W=\R e_1$): We have $E=e^2$,
$\Phi_0=\{e_1\}$, $\Phi^+=\{e_2,e_1+e_2\}$ and $\Phi^-=\{e_1-e_2\}$.
\[\begin{array}{ll}
v(\Phi,\c_2)(a)-v(\Phi,\c_3)(a)&={\rm Pol}(1,\Phi\setminus
\Phi_0,E)(a)\\
&=\Res_{z=0} \left(\frac{e^{a_1x_1+a_2(x_2+z)
}}{(x_2+z)(x_1+x_2+z)(x_1-x_2-z)}\right)_{x=0}=-\frac{1}{2}a_2^2.
\end{array}\]
Using (ii),
$v(\Phi,\c_3)(a)=\frac{1}{2}a_1^2-\frac{1}{4}(a_1-a_2)^2+\frac{1}{2}a_2^2=\frac{1}{4}(a_1+a_2)^2$.
\end{example}

\section{Wall crossing formula  for the partition function: unimodular case}

In this section, we compute the jump $k(\Phi,\c_1)-k(\Phi,\c_2)$ of
the partition function $k(\Phi)$ across a wall. In order to outline
the main ideas in the proof, we will first consider the case where
$\Phi$ is unimodular (see Remark~\ref{uni} for the definition). We
give two formulae. The first one is the convolution formula of
Paradan (\cite{Par}, Theorem 5.2). The second one is a one
dimensional residue formula.

\subsection{Discrete convolution}\label{sec:convdic}
Let $E\in U$ be a primitive element with respect to $\Gamma^*$ so
that $\langle E,\Gamma\rangle  = \Z$. Let $\Psi$ be a sequence of
vectors in $\Gamma$ such that $\langle \psi,E\rangle \neq 0$ for
all $\psi\in \Psi$. Thus $\Psi=\Psi^+\cup \Psi^-$ with
$\Psi^+=\Psi\cap V^+$ and $\Psi^-=\Psi\cap V^-$. Define
$$R_+(\Psi)=[\psi \|\psi \in \Psi^+]\cup [-\psi|\psi\in
\Psi^-].$$ Let  $$W:=\{a\in V\,|\, \langle a,E\rangle =0\},$$
 $$\Gamma_0=\Gamma\cap W.$$
We choose $F\in \Gamma$ such that $\langle E,F\rangle =1$. We thus have
$\Gamma=\Gamma_0\oplus \Z F$.

Let $\Gamma_{\geq 0}$  be the set of  elements $ a \in \Gamma$ such
that $\langle  a ,E\rangle \geq 0$.  Let us define the function $K^+(\Psi)$ on
$\Gamma_{\geq 0}$ such that we have, for $x\in C(R_+(\Psi))^*$,
\begin{equation}\label{eqF}
\prod_{\psi\in \Psi}\frac{1}{1-e^{-\langle \psi,x\rangle }}=\sum_{a \in
\Gamma_{\geq 0}}K^+(\Psi)(a) e^{-\langle a,x\rangle },
\end{equation}
that is  we have written

$$1/(1-e^{-\psi})=\sum_{n\geq 0}e^{-n\psi}, \,{\rm if}\, \psi \in \Psi^+$$
 and
$$1/(1-e^{-\psi})=-e^{\psi}/(1-e^{\psi})
=-\sum_{n> 0}e^{n\psi}\,{\rm if}\, \psi \in \Psi^-.$$

Let $\kappa_-=\sum_{\psi\in \Psi^-}\psi$ so that $\langle \kappa_-,E\rangle =
\sum_{\psi\in \Psi^-}\langle \psi,E\rangle $ is a strictly negative number if
and only $\Psi^-$ is non empty. Then we have  $$K^+(\Psi)(a)=(-1)^{|\Psi_-|}
k(R_+(\Psi))(a-\kappa_-),$$ where $k(R_+(\Psi))$  is the partition function of the system $R_+(\Psi)$.

The function $K^{+}(\Psi)$ is supported on the pointed cone
$-\kappa_-+C(R_+(\Psi))$. In particular the value $K^+(\Psi)(0)$ is
$1$ if $\Psi^-$ is empty, or $0$ if $\Psi^-$ is not empty.

\bigskip
Let $q$ be a   polynomial function on $\Gamma_0$.  Define for
$a\in \Gamma$
$$C(q,\Psi,E)(a):=\sum_{ w \in \Gamma_0}
q(w ) K^+(\Psi)(a- w).$$

The sum is over the finite set $\Gamma_0\cap (a-C(R_+(\Psi)))$.

\begin{theorem}\label{convol}
Assume $\Psi^+$ is non empty. Assume that, for any $\psi\in \Psi$,
 we have $\langle \psi,E\rangle =\pm 1$. Let $q$ be a  polynomial function on
$\Gamma_0$. Then, for $a\in \Gamma_{\geq 0}$,
$$C(q,\Psi,E)(a) ={\rm  Par}(q,\Psi,E)(a).$$
\end{theorem}

\begin{proof}
We need to compute,  for $a\in \Gamma_{\geq 0}$,
$$S(a):=\sum_{ w \in \Gamma_0}q(w ) K^+(\Psi)(a-w).$$
This sum is over a finite set.

Let  $Q\in S(U)$ be  a polynomial function on $V$ extending $q$. We
may write
$$S(a)= \left( Q(\partial_x)\cdot \sum_{ w \in \Gamma_0}K^+(\Psi)(a-w ) e^{\langle  w ,x \rangle}\right)_{x=0}.$$
Define $$G_{x}(a)=\sum_{ w \in \Gamma_0}K^+(\Psi)(a-w)
 e^{-\langle a- w ,x \rangle }.$$
Then $G_{x}(a)$ depends in an analytic way of the variable $x\in U$, and we have

\begin{equation}\label{myS}
S(a)= \left(Q(\partial_x)\cdot e^{\langle a,x\rangle
}G_{x}(a)\right)|_{x=0}.
\end{equation}

The function  $a\mapsto G_{x}(a)=\sum_{ w \in
\Gamma_0}K^+(\Psi)(a-w)e^{-\langle a- w ,x \rangle }$  is a function
on $\Gamma/\Gamma_0=\Z F$. To identify the function $G_{x}(nF)$,  we
compute its discrete Laplace transform in  one variable. Let $x$ be
in $C(R_+(\Psi))^*$. Write $u=e^{-z}$. We compute

\[\begin{array}{ll}
L_{\rm dis}(G_x)(u)&=\sum_{n\geq 0} G_x(nF)u^n\\
&=\sum_{n\geq 0} G_x(nF)e^{-nz}\\
&=\sum_{n,w}K^+(\Psi)(nF-w)e^{-\ll nF-w,x\rr}e^{-\ll nF,zE\rr}\\
&=\sum_{a\in \Gamma}K^+(\Psi)(a)e^{-\ll a,x\rr}e^{-\ll a,zE\rr}\\
&=\frac{1}{\prod_{\psi\in \Psi} (1-e^{-\langle \psi,x+ zE \rangle
})}\\
&=\frac{1}{\prod_{\psi\in \Psi} (1-e^{-\langle
\psi,x\rangle }u^{\ll \psi, E\rr})}.\end{array}\]

In the fourth equality, we have written any element $a\in
\Gamma_{\geq 0}$ as $a=nF-w$, with $n\geq 0$ and $w\in \Gamma_0$.
The next equality is by the definition of the function
$K^+(\Psi)(a)$. Furthermore, we see that the sum is convergent
when $|u|<1$.

As  $L_{\rm dis}(G_x)(u)=\sum_{n\geq 0} G_x(nF)u^n$, Cauchy formula reads
\[\begin{array}{ll}
G_x(nF)&=\frac{1}{2i\pi} \int_{|u|=\epsilon} u^{-n}L_{\rm dis}(G_x)(u)\frac{du}{u}\\
&=\frac{1}{2i\pi} \int_{|u|=\epsilon}
\frac{u^{-n}} {\prod_{\psi\in \Psi}
(1-e^{-\langle \psi,x\rangle}u^{\ll\psi,E\rr})} \frac{du}{u}.
\end{array}\]

Thus we obtain  for $n=\ll a,E\rr$,
$$S(a)=\left(Q(\partial_x)\cdot e^{\langle a,x\rangle}
 \frac{1}{2i\pi} \int_{|u|=\epsilon}
\frac{u^{-n}} {\prod_{\psi\in \Psi}
(1-e^{-\langle \psi,x\rangle}u^{\ll\psi,E\rr})} \frac{du}{u}\right)_{x=0}$$
$$=\frac{1}{2i\pi} \int_{|u|=\epsilon}u^{-n}
\left( Q(\partial_x)\cdot\frac{e^{\langle a,x\rangle }}
 {\prod_{\psi\in \Psi}
(1-e^{-\langle \psi,x\rangle}u^{\ll \psi,E\rr})}\right)_{x=0} \frac{du}{u}.$$

As  at least  one of the $\ll \psi,E\rr$ is positive and $n\geq 0$,
it is easy to see that the function under the integrand has no pole
at $u=\infty$. As $\ll \psi,E\rr=\pm 1$, its poles are obtained for
$u=0$ and $u=1$. The integral on $|u|=\epsilon$ computes the residue
at $u=0$. We use the residue theorem so that $-S(a)$ can also be
computed  as the  residue for $u=1$. We use the coordinate
$u=e^{-z}$ near $u=1$, and we obtain
$$S(a)=\Res_{z=0} \left(Q(\partial_x)
\cdot \frac{e^{\langle a,x +zE\rangle}} {\prod_{\psi\in \Psi}
(1-e^{-\langle \psi,x+zE\rangle})}\right)_{x=0}$$ which establishes
the formula in the theorem.
\end{proof}\bigskip

Finally, we compute the restriction to $W\cap \Gamma$ of ${\rm
Par}(q,\Psi,E)$.

\begin{lemma}\label{restrictionuni}
\begin{itemize}
\item
If $|\Psi^-|=\emptyset$, then the restriction of ${\rm
Par}(q,\Psi,E)$ to $W$ is equal to $q$.
\item If $|\Psi^-|>0$, then the restriction of ${\rm Par}(q,\Psi,E)$
to $W$ vanishes.
\end{itemize}
\end{lemma}
\begin{proof}
The sum  formula gives ${\rm Par}(q,\Psi,E)(w)= K^+(\Psi)(0)q(w)$.
Recall that $K^+(\Psi)(0)$ vanishes as soon as $|\Psi^-|>0$; it is
equal to $1$ if $\Psi^-=\emptyset$.
\end{proof}\bigskip

\subsection{The jump for the partition function}\label{sec:jumpuni}
Let $\Phi$ be a sequence of vectors  spanning the lattice $\Gamma$.
In this section we assume that $\Phi$ is unimodular and that
$\Gamma=\Z\Phi$.

Let $k(\Phi)(a)$ be the partition function. Then $k(\Phi)(a)$
coincides with  a polynomial function on each  chamber. We consider,
as in Section  \ref{sec:jumpvol}, two adjacent chambers $\c_1$ and
$\c_2$ separated by a wall $W$.  As before, $\Phi_0$ denotes $W\cap
\Phi$; it is also a unimodular system for  the lattice $\Gamma\cap
W$. Let $k_{12}=k(\Phi_0,\c_{12})$ be the polynomial function on
$W\cap \Gamma$ associated to the chamber $\c_{12}$ of $\Phi_0$.
Consider the sequence $\Psi=\Phi\setminus \Phi_0$. We choose $E\in
U$ such that $\Psi^+$ is non empty.

As the system $\Phi$ is assumed to be unimodular, the integers
$d_\phi=\langle \phi,E\rangle $ are equal  to $\pm 1$ for any
$\phi\in \Phi$ not in $W$.

\begin{theorem}\label{jumpkuni}
Let $k_{12}=k(\Phi_0,\c_{12})$ be the  polynomial function on $W$
associated to the chamber $\c_{12}$. Then,  if $\ll E,\c_1\rr>0$, we
have
\begin{equation}\label{Leqdis}
k(\Phi,\c_1)-k(\Phi,\c_2)={\rm Par}(k_{12},\Phi\setminus
\Phi_0,E).
\end{equation}
\end{theorem}

\begin{remark}
By Theorem \ref{convol}, the function ${\rm
Par}(k_{12},\Phi\setminus \Phi_0,E)$ coincide, up to sign, on
$\Gamma_{\geq 0}$ with the discrete convolution $\sum_{w\in W\cap
\Gamma}k_{12}(w)k(R_+(\Psi))(a-\kappa_--w)$ of the polynomial
function $k_{12}(w)$ on $W$ by the partition function (shifted)
$k(R_+(\Psi))$.  Thus,  our residue  formula for
$k(\Phi,\c_1)-k(\Phi,\c_2)$ coincide with Paradan's formula
(\cite{Par}, Theorem 5.2) for the jump of the partition function.
\end{remark}

\begin{proof}
Denote by ${\rm Leq}(\Phi)$ the left
hand side and by ${\rm Req}(\Phi)$  the right hand side of Equation
(\ref{Leqdis}) above.

 As
the system is unimodular, the lattice $\Gamma$ is generated by
$\Gamma_0$ and any $\phi$ not in $W$. So in a parallel way to the
proof of the jump for the volume function, it would be sufficient
to verify that  $D(\phi) ({\rm Leq}(\Phi)-{\rm Req}(\Phi))=0$ for
some $\phi\in\Phi$ not in $W$, and  that  $({\rm Leq}(\Phi)-{\rm
Req}(\Phi))$ vanishes on $W$. However, in order that our proof
adapts without change to the non unimodular case, we will check
$D(\phi) ({\rm Leq}(\Phi)-{\rm Req}(\Phi))=0$ for any
$\phi\in\Phi$.

Let us first verify   Equation (\ref{Leqdis}) above  when  there
is only one vector $\phi$ of $\Phi$ not in $W$.  We can suppose
that $\Psi^+=\{\phi\}$ and $\Psi^-$ is empty. In this case the
wall $W$ is a facet of the cone $C(\Phi)$. The chamber $\c_1$ is
equal to $\c_{12}\times \R_{>0} \phi$, while $\c_2$ is the
exterior chamber. It is easy to see that $k(\Phi,\c_1)(w+t
\phi)=k_{12}(w)$, whereas $k(\Phi,\c_2)=0$. The equation follows
from the first item of Lemma \ref{restrictionV}.

Suppose that this is not the case. Let $\phi$ be in $\Phi$, and
denote $\Phi'=\Phi\setminus \{\phi\}$. We study the difference
equations satisfied by ${\rm Leq}(\Phi)$ and ${\rm Req}(\Phi)$.

We have several cases to consider.
\begin{itemize}
\item  $\phi$ is not in $W$.

Then  the  sequence  $\Phi'=\Phi\setminus \{\phi\}$ spans $V$ and
$W$  is a wall for $\Phi'$. The intersection of $\Phi'$ with $W$ is
$\Phi_0$.  Let $\c'_1$ and $\c'_2$ be the chambers for $\Phi'$
containing $\c_1, \c_2$. Then, they are adjacent with respect to
$W$. The chamber $\c_{12}$ remains the same. By Lemma
\ref{lem:diffkc}, we have
$$D(\phi) (k(\Phi,\c_1)-k(\Phi,\c_2))=k(\Phi',\c'_1)-k(\Phi',\c'_2).$$
By Proposition \ref{diffpar},
$$D(\phi) {\rm Par}(k_{12}, \Phi\setminus \Phi_0,E)={\rm Par}(k_{12},\Phi'\setminus\Phi_0,E).$$
By induction, we obtain $D(\phi) ({\rm Leq}(\Phi)-{\rm Req}(\Phi))=0$.

\item $\phi$ is in $W$ and $\Phi'_0=\Phi_0\setminus \{\phi\}$ span $W$.

Then  the  sequence $\Phi'=\Phi\setminus \{\phi\}$ spans $V$, and
$W$ is a wall for the system $\Phi'$. Let $\c'_1$ and $\c'_2$ be the
chambers for $\Phi'$ containing $\c_1,\c_2$. Then, they are adjacent
with respect to $W$. By Lemma \ref{lem:diffkc}, we have
$$D(\phi) (k(\Phi,\c_1)-k(\Phi,\c_2))=k(\Phi',\c'_1)-k(\Phi',\c'_2).$$

Let $\c'_{12}$ be the chamber for the sequence $\Phi'_0$ containing
$\c_{12}$.  The sequence $\Phi\setminus \Phi_0$ is equal to
$\Phi'\setminus \Phi'_0$. By Proposition \ref{diffpar}, we obtain
\[\begin{array}{ll}
D(\phi){\rm Par}(k(\Phi_0,\c_{12}), \Phi\setminus \Phi_0,E)&= {\rm Par}(k(\Phi'_0,\c'_{12}),\Phi\setminus \Phi_0,E)\\
&={\rm Par}(k(\Phi'_0,\c'_{12}),\Phi'\setminus \Phi'_0,E).\end{array}\]
So by induction, we conclude that
$D(\phi) ({\rm Leq}(\Phi)-{\rm Req}(\Phi))=0$ again.\\

\item $\phi$ is in $W$ and $\Phi'_0$ does not span $W$.
Then $W$ is not a wall for the sequence $\Phi'$.

It follows from the description given in Proposition
\ref{qualitative} of the regular  behavior of functions on chambers
that $k(\Phi',\c'_{1})-k(\Phi',\c'_2)=0$. Thus, by  Lemma
\ref{lem:diffkc},
$$D(\phi) (k(\Phi,\c_1)-k(\Phi,\c_2))=k(\Phi',\c'_1)-k(\Phi',\c'_2)=0.$$

Similarly, the function $k(\Phi_0,\c_{12})$ satisfies
$D(\phi)k(\Phi_0,\c_{12})=0$. As $$D(\phi){\rm
Par}(k(\Phi_0,\c_{12}), \Phi\setminus \Phi_0,E)= {\rm
Par}(D(\phi)k(\Phi_0,\c_{12}), \Phi\setminus \Phi_0,E),$$ we obtain
$D(\phi) ({\rm Leq}(\Phi))=0= D(\phi)({\rm Req}(\Phi))$.
\end{itemize}

We conclude that $D(\phi) ({\rm Leq}(\Phi)-{\rm Req}(\Phi))=0$ for
any $\phi\in \Phi$ so that ${\rm Leq}(\Phi)-{\rm Req}(\Phi)$ is
constant on $\Gamma=\Z\Phi$. It is thus sufficient to verify that
${\rm Leq}(\Phi)-{\rm Req}(\Phi)$  vanishes on $W$. If $\Phi^+$ and
$\Phi^-$ are both non empty, both chambers $\c_1$ and $\c_2$ are
interior chambers. So, ${\rm Leq}(\Phi)$ vanishes on $W$. By Lemma
\ref{restrictionuni}, ${\rm Req}(\Phi)$ also vanishes on $W$. If
$\Phi^-$ is empty, the wall $W$ is a facet of $C(\Phi)$. The
restriction of the function $k(\Phi,\c)$ to a facet is the function
$k(\Phi_0,\c_{12})$. Thus ${\rm Leq}(\Phi)$ restricts to $k_{12}$ on
$W$ by Lemma \ref{face}. This is the same for ${\rm Req}(\Phi)$ by
Lemma \ref{restrictionuni}. Thus we established the theorem.
\end{proof}\bigskip

\begin{figure}[!h]
\begin{center}\includegraphics{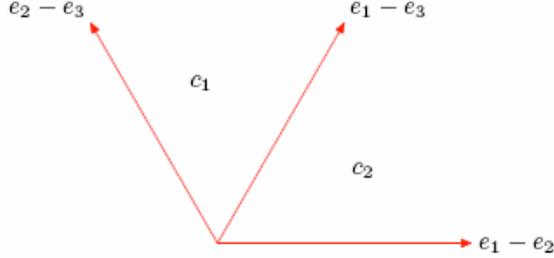}
\end{center}\caption{Chambers of $A_2$}\label{a2}
\end{figure}

\begin{figure}
\begin{center}
  \includegraphics[width=60mm]{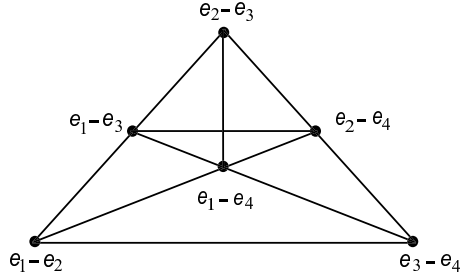}\\
  \caption{Chambers of $A_3$}\label{a3}
\end{center}
\end{figure}

We now give some examples of jumps in partition functions for the
root system of type $A_r$ (which is unimodular).

Let $Y$ be an $(r+1)$ dimensional vector space with basis $\{e_i,
i:1,\ldots, r+1\}$; we denote its dual basis by $\{e^i\}$. Let $V$
denote the vector space generated by the set of positive roots
\[\Phi(A_r)=\{e_i-e_j: \;1 \leq i <j \leq r+1\}\] of $A_r$. Then, $V$
is a hyperplane in $Y$ formed by points $v=\sum_{i=1}^{r+1}v_ie_i
\in Y$ satisfying $\sum_{i=1}^{r+1}v_i=0$. Using the explicit
isomorphism $p:\mathbb{R}^r \to V$ defined by $(a_1,\ldots,
a_r)\mapsto a_1e_1+\cdots+a_re_r-(a_1+\cdots+a_r)e_{r+1}$, we
write $a \in V$ as $a=\sum_{i=1}^r a_i(e_i-e_{r+1})$. Under $p^*$,
the vector $e_i-e_{r+1}$ determines the linear function $x_i$ in
$U=V^*\sim \mathbb{R}^r$.  The vector $a\in V$ lies in
$C(\Phi(A_r))$ if and only if $a_1+\cdots+a_i\geq 0$ for all
$i:1,\ldots,r$.  This will be our notation for subsequent examples
concerning $A_r$.

\begin{example}\label{ka2} We consider the root system of type $A_2$ with
$\Phi=\{e_1-e_2,e_2-e_3,e_1-e_3\}$ (see Figure ~\ref{a2}).
   The
cone $C(\Phi)$ is comprised of two chambers
$\c_1=C(\{e_1-e_3,e_2-e_3\})$ and $\c_2=C(\{e_1-e_3,e_1-e_2\})$.
We will calculate both $k(\Phi,\c_1)$ and $k(\Phi,\c_2)$ using our
formula in Theorem \ref{jumpkuni} iteratively starting from an
exterior chamber.

(i) Jump from the exterior chamber to $\c_1$: In this case, $E=e^1$,
$\Phi_0=\{e_2-e_3\}$, $\Phi^+=\{e_1-e_3,e_1-e_2\}$ and
$\Phi^-=\emptyset$.
\[\begin{array}{ll}
k(\Phi,\c_1)(a)-k(\Phi,\c_{\rm ext})(a)&={\rm Par}(1,\Phi \setminus \Phi_0,e^1)(a)\\
k(\Phi,\c_1)(a)&= \Res_{z=0} \left(
\frac{e^{a_1x_1+a_2x_2+za_1}}{(1-e^{-x_1-z})(1-e^{-x_1+x_2-z})}\right)_{x=0}\\
&=\Res_{z=0}\frac{e^{za_1}}{(1-e^{-z})^2}=1+a_1.
\end{array}\]

(ii) Jump from $\c_1$ to $\c_2$: We have $E=e^2$,
$\Phi_0=\{e_1-e_3\}$, $\Phi^-=\{e_1-e_2\}$ and $\Phi^+=\{e_2-e_3\}$.
\[\begin{array}{ll}
k(\Phi,\c_1)(a)-k(\Phi,\c_2)(a)&={\rm Par}(1,\Phi \setminus \Phi_0,e^2)(a)\\
&= \Res_{z=0} \left(
\frac{e^{a_1x_1+a_2x_2+za_2}}{(1-e^{-x_2-z})(1-e^{-x_1+x_2-z})}\right)_{x=0}\\
&=\Res_{z=0}\frac{e^{za_2}}{(1-e^{-z})(1-e^{z})}=-a_2.
\end{array}\]
Then, $k(\Phi,\c_2)(a)=1+a_1+a_2$.
\end{example}

\begin{example}\label{ka3} We now consider the root system of type $A_3$ (see
Figure~\ref{a3} which depicts the $7$ chambers of $A_3$ via the
intersection of the ray $\R^+\alpha$ of each root $\alpha$ with
the plane $3a_1+2a_2+a_3=1$).

We will calculate the jump in the partition function from
$\c_1:=C(\{e_1-e_2,e_1-e_3,e_1-e_4\})$ to
$\c_2:=C(\{e_1-e_2,e_3-e_4,e_1-e_4\})$. In this case, $E=e^3$,
$\Phi_0=\{e_1-e_4, e_2-e_4,e_1-e_2\}$, $\Phi^+=\{e_3-e_4\}$ and
$\Phi^-=\{e_1-e_3,e_2-e_3\}$.  Notice that $k_{12}$ is the partition
function corresponding to the chamber $C(\{e_1-e_4,e_1-e_2\})$ of
the copy of $A_2$ in $A_3$ having the set $\Phi_0$ as its set of
positive roots.   Using the final calculation in part (ii) of
Example~\ref{ka2} ($e_4$ here plays the role of $e_3$ in that
example), $k_{12}(a)=a_1+a_2+1$. Then, by Theorem \ref{jumpkuni},
\[\begin{array}{ll}
k(\Phi,\c_2)(a)-k(\Phi,\c_1)(a)&={\rm Par}(k_{12},\Phi \setminus \Phi_0,E)(a)\\
&= \Res_{z=0} \left((\partial_{x_1}+\partial_{x_2}+1)\cdot
\frac{e^{a_1x_1+a_2x_2+za_3}}{(1-e^{-z})(1-e^{-x_1+z})(1-e^{-x_2+z)})}\right)_{x=0}\\
&=\frac{1}{6}a_3(a_3-1)(2a_3+3a_2+3a_1+5).
\end{array}\]
\end{example}

\subsection{Khovanskii-Pukhlikov differential operator}\label{sec:kho}
 We recall that $\Gamma=\Z\Phi$. We normalize the measure $dx$ in order that it gives volume $1$ to a fundamental domain for $\Gamma^*$, and we write $v(\Phi,dx,\c)$ simply as
 $v(\Phi,\c)$.

We recall the relation between the function $k(\Phi,\c)$ and
$v(\Phi,\c)$. Define ${\rm Todd}(z)$ as the expansion of
$$\frac{z}{1- e^{-z}}=1+\frac{1}{2}z+\frac{1}{12}z^2+\cdots$$
in power series in $z$. For $\phi\in \Phi,$ ${\rm
Todd}(\partial(\phi))$ is a differential operator of infinite
order with constant coefficients.  If $p$ is a polynomial function
on $V$,  ${\rm Todd}(\partial(\phi))p$ is well defined and is a
polynomial on $V$. We denote by ${\rm Todd}(\Phi,\partial)$ the
operator defined on  polynomial functions on $V$ by
$${\rm Todd}(\Phi,\partial)=
\prod_{\phi\in \Phi}{\rm Todd}(\partial(\phi)).$$

The operator ${\rm Todd} (\Phi,\partial)$ transforms a polynomial
function into a  polynomial function on $\Gamma$.

The following result has been proven in  Dahmen-Micchelli
\cite{dm}.

\begin{theorem}\label{todduni}
Let $\c$ be a chamber. Then $$k(\Phi,\c)(a)= {\rm
Todd}(\Phi,\partial)\cdot v(\Phi,\c).$$
\end{theorem}

Here we give yet another proof of this theorem, by verifying that
our explicit formula for the jumps are related by the Todd
operator.

Let $W$ be a wall of $\Phi$ determined by $E$. Assume that
$\Phi^+$ is non empty. Let ${\rm Todd}(\Phi_0,\partial)$ be the
Todd operator related to the sequence $\Phi_0=\Phi\cap W$ which is also
unimodular.

\begin{proposition}\label{toddpol}
Let $p$ be a polynomial function on $W$. Then
$${\rm Todd}(\Phi,\partial){\rm Pol}(p,\Phi\setminus
\Phi_0,E)={\rm Par}({\rm Todd}(\Phi_0,\partial)P,
\Phi\setminus\Phi_0,E).$$
\end{proposition}

\begin{proof}
We have $${\rm Todd}(\Phi,\partial)=  {\rm Todd}(\Phi\setminus \Phi_0,\partial) {\rm Todd}(\Phi_0,\partial).$$
By  Proposition \ref{diffPv}
$${\rm Todd}(\Phi_0,\partial)
{\rm Pol}(p,\Phi\setminus
\Phi_0,E)={\rm Pol}({\rm Todd}(\Phi_0,\partial)p,
\Phi\setminus\Phi_0,E).$$

Let $q={\rm Todd}(\Phi_0,\partial)p$.
Then we apply ${\rm Todd}(\Phi\setminus \Phi_0,\partial)$ to $${\rm
Pol}(q,\Phi\setminus \Phi_0,E)=\Res_{z=0} \left(Q(\partial_x)\cdot
\frac{e^{\langle a,x+z E\rangle }}{\prod_{\phi\in \Phi\setminus \Phi_0}\langle
\phi,x+zE\rangle }\right)_{x=0}.$$
We obtain
$${\rm Todd}(\Phi\setminus \Phi_0,\partial){\rm Pol}(q,\Phi\setminus \Phi_0,E)(a)=  {\rm Par}(q,\Phi\setminus \Phi_0,E).$$
\end{proof}\bigskip
We now prove Theorem \ref{todduni} by induction.  We assume that
${\rm Todd}(\Phi_0,\partial) v(\Phi_0,\c_{12})= k(\Phi_0,\c_{12})$.
We then obtain from Proposition \ref{toddpol}:
\[\begin{array}{ll}
{\rm Todd}(\Phi,\partial)( v(\Phi,\c_1)-v(\Phi,\c_2))&=
{\rm Todd}(\Phi,\partial){\rm Pol}(v(\Phi_0,\c_{12}),\Phi\setminus \Phi_0,E)\\
&={\rm Par}({\rm Todd}(\Phi_0,\partial) v(\Phi_0,\c_{12}), \Phi\setminus \Phi_0,E)\\
&= {\rm Par}(k(\Phi_0,\c_{12}), \Phi\setminus \Phi_0,E)\\
&=k(\Phi,\c_1)-k(\Phi,\c_2).\end{array}\]

Starting from the exterior chamber where $k(\Phi,\c_{\rm ext})={\rm
Todd}(\Phi,\partial)\cdot v(\Phi,\c_{\rm ext})=0$, we obtain by
jumping over the walls  that $k(\Phi,\c)={\rm
Todd}(\Phi,\partial)\cdot v(\Phi,\c)$ for any chamber.

\section{Wall crossing formula for the partition function: general case}\label{sec:jumpgeneral}

In this section, we compute the jump $k(\Phi,\c_1)-k(\Phi,\c_2)$ of
the partition function $k(\Phi)$ across a wall when $\Phi$ is an
arbitrary system.

\subsection{A particular quasi-polynomial function}
Let $W$ be a hyperplane of $V$ determined by a primitive vector $E$, and
$\Gamma_0=W\cap \Gamma$.
We denote by $T$ the torus $V^*/ \Gamma^*$ and $T_0$ the torus $W^*/\Gamma_0^*$.
The restriction map $V^*\to W^*$ induces a surjective homomorphism
$r: T\to T_0$.  The kernel of $r$ is isomorphic to $\R/\Z$.

Let $Q$ be a quasi-polynomial function on $\Gamma$. We may write
$Q(a)=\sum_{y\in T}e_y(a)Q_y(a)$ where $y\in V^*$  give rise to an
element of finite order in $V^*/\Gamma^*$, still denoted by $y$. The
set of elements $g\in  T$ such that $r(g)=r(y)$ is isomorphic to
$\R/\Z$.
\begin{definition}\label{defPolquasi}
Let $Q(a)=\sum_{y\in T}e_y(a)Q_y(a)$ be a quasi-polynomial function
on $\Gamma$ and let $\Psi$ be a sequence of vectors not belonging to
$W$. We define, for $a\in \Gamma$,
$${\rm Para}(Q,\Psi,E)(a)=\sum_{y\in T} \sum_{g\in T|r(g)=r(y)}\Res_{z=0}\left(Q_y(\partial_x)\cdot
\frac{e^{\langle a,x+2i\pi g+ z E\rangle }}{\prod_{\psi\in
\Psi}(1-e^{-\langle \psi,x+2i\pi g+zE\rangle}) }\right)_{x=0}.$$
\end{definition}
\begin{remark}\label{rem:exp}
The definition may look strange, as we sum a priori on the infinite set $r(g)=r(y)$. However, in order that the function
$$\left(Q_y(\partial_x)\frac{e^{\langle a,x+2i\pi g+ z E\rangle }}{\prod_{\psi\in
\Psi}(1-e^{-\langle \psi,x+2i\pi g+zE\rangle}) }\right)_{x=0}$$ has
a pole at $z=0$,  we see that there must exist a $\psi$ in $\Psi$
such that $e^{2i\pi \ll\psi,g\rr}=1$. As $r(g)$ is fixed, this
leaves a finite number of possibilities for $g$. More concretely, if
$y$ is given, any $g$ such that $r(g)=r(y)$ is of the form $g=y+GE$,
and $G$ must satisfy $e^{2i\pi G\ll \psi,E\rr}=e^{-2i\pi \ll
y,\psi\rr}$ for some $\psi\in \Psi$. Furthermore, we see that if the
integers $\ll \psi,E\rr$ are equal to $\pm 1$ for all $\psi\in
\Psi$, and if $Q$ is polynomial (so that $y=0$ on the above
equation), then  ${\rm Para}(Q,\Psi,E)$ is equal to ${\rm
Par}(Q,\Psi,E)$.
\end{remark}
\medskip
It is easy to see  that  ${\rm Para}(Q,\Psi,E)(a)$ is a
quasi-polynomial function of $a\in V$. Furthermore,  using the same
argument as in the proof of Lemma \ref{lem:Pp}, we obtain the
following.

\begin{lemma}
The quasi-polynomial function ${\rm Para}(Q,\Psi,E)$ depends only on
the restriction $q$ of $Q$ to $\Gamma_0$.
\end{lemma}

Choose a primitive vector $F$ such that $\Gamma=\Gamma_0\oplus \Z
F$. Then we see that any quasi-polynomial function $q$ on $\Gamma_0$
extends to a quasi-polynomial function $Q$ on $\Gamma$.

\begin{definition}
Let $q$ be a quasi-polynomial function on $\Gamma_0$. We define
$${\rm Para}(q,\Psi,E):={\rm Para}(Q,\Psi,E),$$ where $Q$ is any quasi-polynomial function on $\Gamma$ extending $q$.
\end{definition}

\begin{remark}
Let $q$ be a quasi-polynomial function on $\Gamma_0$; we may write
$q(w)=\sum_{y\in T_0}e_y(w)q_y(w)$.  Let $Q_y$ denote any extension
of the polynomial function $q_y$ on $V$.  Then, while calculating
${\rm Para}(q,\Psi,E)$, we are in fact summing over $g \in T$ such
that $r(g)=y$.
\end{remark}

\begin{proposition}\label{pro:difQp}
Let $\psi\in \Psi$. Then
$$D(\psi) {\rm Para}(q, \Psi,E)={\rm Para}(q, \Psi\setminus \{\psi\},E).$$
Let $w\in \Gamma_0$. Then,
$$D(w) {\rm Para}(q,\Psi,E)={\rm Para}(D(w)q,\Psi,E).$$
\end{proposition}
\begin{proof}
The first formula is immediate from the definition.
For the second formula, if $r(g)=r(y)$ and
if $w\in \Gamma_0$, then $$ e^{\ll a-w, x+2i\pi g+zE\rr}= e^{-\ll w,x\rr} e_y(-w) e^{\ll a, x+2i\pi g+zE\rr},$$  and the result follows as in the proof of Proposition \ref{diffpar}.
\end{proof}

\subsection{Discrete convolution}
We take the same notations as in Section \ref{sec:convdic}.  However here the
system $\Psi$ is arbitrary.
We define, as before, the function $K^+(\Psi)$ on $\Gamma_{\geq 0}$
by the equation
\begin{equation}
\prod_{\psi\in \Psi}\frac{1}{1-e^{-\langle \psi,x\rangle }}=\sum_{a \in
\Gamma_{\geq 0}}K^+(\Psi)(a) e^{-\langle a,x\rangle }.
\end{equation}

Let $q$ be a quasi-polynomial function on $\Gamma_0$. Define for
$a\in \Gamma_{\geq 0}$
$$C(q,\Psi,E)(a):=\sum_{ w \in \Gamma_0}
q(w ) K^+(\Psi)(a- w).$$

\begin{theorem}
Let $q$ be a quasi-polynomial function  on $\Gamma_0$. Assume that
$\Psi^+$ is non empty. Then, for $a\in \Gamma_{\geq 0}$,
$$C(q,\Psi,E)(a) = {\rm Para}(q,\Psi,E)(a).$$
\end{theorem}

\begin{proof}
We need to compute,  for $a\in \Gamma_{\geq 0}$,
$$S(a):=\sum_{ w \in \Gamma_0}q(w ) K^+(\Psi)(a-w).$$
This sum is over a finite set.

Let $Q(a)=\sum_{y\in T} e_y(a)  Q_y(a)$  be any quasi-polynomial
function on $\Gamma$ extending $q$. We may write
$$S(a)=\sum_{y \in T} \left(Q_{y}(\partial_x)\cdot \sum_{ w \in \Gamma_0}K^+(\Psi)(a-w ) e^{\langle  w ,x+2i\pi y \rangle}\right)_{x=0}.$$
Define $$G_{x,y}(a)=\sum_{ w \in \Gamma_0}K^+(\Psi)(a-w)
 e^{-\langle a- w ,x+2i\pi y \rangle }.$$
Then $G_{x,y}(a)$ depends in an analytic way of the variable $x\in U$, and we have
\begin{equation}\label{mySg}
S(a)=\sum_{y\in T} \left(Q_{y}(\partial_x)\cdot e^{\langle a,x+2i\pi
y\rangle }G_{x,y}(a)\right)_{x=0}.
\end{equation}
The function  $a\mapsto G_{x,y}(a)=\sum_{ w \in
\Gamma_0}K^+(\Psi)(a-w )
 e^{-\langle a- w ,x+2i\pi y \rangle }$  is a function on $\Gamma/\Gamma_0=\Z F$.
To identify the function $G_{x,y}(n F)$, with $n=\langle
a,E\rangle$, we compute its discrete Laplace transform in one
variable. With the same proof as the proof in Theorem
\ref{convol}, we obtain that for $x$ in the dual cone to
$C(R_+(\Psi))$, $L_{\rm dis}(G_{x,y})(u)=\sum_{n\geq 0}
G_{x,y}(nF)u^n$ is convergent for $|u|<1$ and we obtain
$$G_{x,y}(a)=\frac{1}{2i\pi}\int_{|u|=\epsilon} \frac{u^{-n}}{\prod_{\psi\in \Psi}
(1-e^{-\langle \psi,x+2i\pi y \rangle
}u^{\ll \psi,E\rr })}\frac{du}{u}$$ where $n=\langle a,E\rangle .$

Thus Formula
(\ref{mySg}) becomes
$$S(a)=\sum_{y\in T} \left(Q_{y}(\partial_x)\cdot e^{\langle a,x+2i\pi y\rangle }
\frac{1}{2i\pi}\int_{|u|=\epsilon} \frac{u^{-n}}{\prod_{\psi\in
\Psi} (1-e^{-\langle \psi,x+2i\pi y\rangle
}u^{\ll\psi,E\rr})}\frac{du}{u}\right)_{x=0}.$$

Let $$F_{y}(u)=\left(Q_{y}(\partial_x)\cdot e^{\langle a,x+2i\pi
y\rangle} \frac{u^{-n}}{\prod_{\psi\in \Psi} (1-e^{-\langle
\psi,x+2i\pi y \rangle }u^{\ll \psi,E\rr})}\right)_{x=0}.$$ As at
least one of the $\ll \psi,E\rr$ is strictly positive and $n\geq 0$,
the function $F_y(u)$ has no pole at $\infty$. The  integral over
$|u|=\epsilon$ computes the residue of $F_y(u)$ at $u=0$. All other
poles are such that $ u^{\ll \psi,E\rr}= e^{\langle \psi,2i\pi y
\rangle}$ for some $\psi\in \Psi$, so they are roots of unity
$\zeta=e^{2i\pi G}$ with $G\in \R/\Z$. Any element $g\in T$ with
$r(g)=r(y)$ is of the form $g=y+G E$ with some $G$. We obtain
$$S(a)=-\sum_{y\in T} \sum_{G\in \R/\Z}  \Res_{u=e^{2i\pi G}} F_y(u).$$
We write $u=e^{2i\pi G}e^{-z}$ in the neighborhood of $e^{2i\pi G}$ and we obtain
the formula of the theorem.
\end{proof}\bigskip

Similarly, we compute the restriction of ${\rm Para}(q,\Psi,E)$ to
$W\cap\Gamma$.
\begin{lemma}\label{restriction}
\begin{itemize}
\item
If $|\Psi^-|=\emptyset$, then the restriction of ${\rm
Para}(q,\Psi,E)$ to $W$ is equal to $q$.
\item If $|\Psi^-|>0$, then the restriction of ${\rm Para}(q,\Psi,E)$
to $W$ vanishes.
\end{itemize}
\end{lemma}
\begin{proof}
The sum  formula gives ${\rm Para}(q,\Psi,E)(w)= K^+(\Psi)(0)q(w)$ and
$K^+(\Psi)(0)$ vanishes as soon as $|\Psi^-|>0$.
\end{proof}

\subsection{The jump for the partition function}
Let $\Phi$ be a sequence of vectors  spanning the lattice $\Gamma$.
Let $k(\Phi)(a)$ be the partition function given by quasi-polynomial
functions on chambers. We consider as in Section \ref{sec:jumpuni}
 two adjacent
chambers $\c_1$ and $\c_2$ separated by a wall $W$. As before,
$\Phi_0$ denotes $W\cap \Phi$.  Let $k_{12}=k(\Phi_0,\c_{12})$ be
the quasi-polynomial function on $W\cap \Gamma$ associated to the
chamber $\c_{12}$ of $\Phi_0$. Consider the sequence
$\Psi=\Phi\setminus \Phi_0$. We choose $E\in U$ such that $\Psi^+$
is non empty. In the preceding section, we have associated  a
quasi-polynomial function ${\rm Para}(k_{12}, \Phi\setminus
\Phi_0,E)$ on $\Gamma$ to $\Phi\setminus\Phi_0$, $E$ and $k_{12}$ .
 We recall that
$${\rm Para}(k_{12},\Phi\setminus \Phi_0,E)(a)=\sum_{w\in W\cap \Gamma} k_{12}(w) K^+(\Phi\setminus \Phi_0)(a-w).$$

\begin{theorem}\label{jumpkgen}
Let $k_{12}=k(\Phi_0,\c_{12})$ be the  quasi-polynomial function on
$\Gamma_0$ associated to the chamber $\c_{12}$. Then, if $\ll
E,\c_1\rr>0$, we have
\begin{equation}
k(\Phi,\c_1)-k(\Phi,\c_2)={\rm Para}(k_{12},\Phi\setminus
\Phi_0,E).
\end{equation}
\end{theorem}
\begin{proof}
The proof is exactly the same as in the proof of Theorem
\ref{jumpkuni} corresponding to the unimodular case.
\end{proof}\bigskip

\begin{example}\label{kb2} We consider the root system of type $B_2$ (see
Figure~\ref{b2}). We will calculate $k(\Phi,\c)$ for all chambers
$\c$ using our formula in Theorem \ref{jumpkgen} iteratively
starting from an exterior chamber.

(i) Jump from the exterior chamber to $\c_1$: We have $E=e^1$,
$\Phi_0=\{e_2\}$, $\Phi^+=\{e_1+e_2,e_1,e_1-e_2\}$ and
$\Phi^-=\emptyset$. With the notation of Definition~\ref{defPol},
$Q=1$(hence $y=0$) and $\ll \phi,E \rr=\pm 1$ for all $\phi \in
\Phi\setminus \Phi_0$.  Then, by Remark~\ref{rem:exp}, ${\rm
Para}(1,\Phi\setminus\Phi_0,E)={\rm Par}(1,\Phi\setminus\Phi_0,E)$.
We get
\[\begin{array}{ll}
k(\Phi,\c_1)(a)-k(\Phi,\c_{\rm ext})(a)&=\Res_{z=0} \left(
\frac{e^{\langle a,x+ze^1
\rangle}}{(1-e^{-(x_1+x_2+z)})(1-e^{-(x_1+z)})(1-e^{-(x_1-x_2+z)})}\right)_{x=0}\\
k(\Phi,\c_1)(a)&=\Res_{z=0}
\left(\frac{e^{a_1z}}{(1-e^{-z})^3}\right)=\frac{1}{2}(a_1+2)(a_1+1).
\end{array}\]
(ii) Jump from $\c_1$ to $\c_2$: We have $E=e^1-e^2$,
$\Phi_0=\{e_1+e_2\}$, $\Phi^+=\{e_1,e_1-e_2\}$ and $\Phi^-=\{e_2\}$.
We also have $Q=k_{12}=1$, thus $y=0$.  Then, the set of feasible $g
\in T$ giving a nontrivial residue at $z=0$ for a summand in ${\rm
Para}(1,\Phi\setminus\Phi_0,E)$ and satisfying $r(g)=0$ is
$\{0,\frac{e^1-e^2}{2}\}$. By Theorem \ref{jumpkgen},
\[\begin{array}{ll}
k(\Phi,\c_2)(a)-k(\Phi,\c_1)(a)&={\rm Para}(1,\Phi\setminus\Phi_0,E)(a)\\
&=\sum_{g=0,g=\frac{e^1-e^2}{2}} \Res_{z=0} \left(\frac{e^{\langle
a,x +2i\pi g +zE\rangle }}{\prod_{\phi\in
\Phi\setminus\Phi_0}(1-e^{-\langle
\phi,x+2i\pi g+zE\rangle })}\right)_{x=0}\\
&=\Res_{z=0}
\left(\frac{e^{(a_1-a_2)z}}{(1-e^{-z})(1-e^{-2z})(1-e^{z})}\right)
+\Res_{z=0}\left(\frac{(-1)^{a_1+a_2}e^{(a_1-a_2)z}}{(1+e^{-z})(1-e^{-2z})(1+e^{z})}\right)\\
&=(-1)^{a_1+a_2}\frac{1}{8}-\frac{1}{8}(2a_2^2-4a_1a_2+2a_1^2+1-4a_2+4a_1)\\
&= \begin{cases} -\frac{1}{4}(a_2-a_1)(a_2-a_1-2) & \text{if $a_1+a_2$ even,}\\
-\frac{1}{4}(a_2-a_1-1)^2 & \text{if $a_1+a_2$ odd} \end{cases}\\
&=-\frac{1}{4}(a_2-a_1-1+(\frac{1+(-1)^{a_1+a_2}}{2}))(a_2-a_1-1-(\frac{1+(-1)^{a_1+a_2}}{2})).
\end{array}\]
Using (i), we get
\[\begin{array}{ll}
k(\Phi,\c_2)(a)&=\frac{1}{2}(a_1+2)(a_1+1)+(-1)^{a_1+a_2}\frac{1}{8}-\frac{1}{8}(2a_2^2-4a_1a_2+2a_1^2+1-4a_2+4a_1)\\
&=\frac{1}{4}a_1^2+\frac{1}{2}a_1a_2-\frac{1}{4}a_2^2+a_1+\frac{1}{2}a_2+\frac{7}{8}+(-1)^{a_1+a_2}\frac{1}{8}.
\end{array}\]
(iii) Jump from $\c_2$ to $\c_3$:  We have $E=e^2$,
$\Phi_0=\{e_1\}$, $\Phi^+=\{e_2,e_1+e_2\}$ and $\Phi^-=\{e_1-e_2\}$.
We again have $Q=k_{23}=1$ and $\ll \phi,E \rr=\pm 1$ for all $\phi
\in \Phi\setminus \Phi_0$.  Thus, $y=g=0$ and (as in part (i)) we
can use the formula for the unimodular case:
\[\begin{array}{ll}
k(\Phi,\c_2)(a)-k(\Phi,\c_3)(a) &=\Res_{z=0} \left( \frac{e^{\langle
a,x +zE\rangle }}{\prod_{\phi\in \Phi\setminus\Phi_0}(1-e^{-\langle
\phi,x+zE\rangle })}\right)_{x=0}\\
&=\Res_{z=0}
\left(\frac{e^{a_2z}}{(1-e^{-z})(1-e^{-z})(1-e^{z})}\right)=-\frac{1}{2}a_2(a_2+1).
\end{array}\]
Then,
$k(\Phi,\c_3)(a)=\frac{1}{4}a_1^2+\frac{1}{2}a_1a_2+\frac{1}{4}a_2^2+a_1+a_2+\frac{7}{8}+(-1)^{a_1+a_2}\frac{1}{8}.$
\end{example}

\subsection{Generalized Khovanskii-Pukhlikov differential operator}
Here $\Phi$ is a general sequence (not necessarily unimodular). We assume again that
$\Z\Phi=\Gamma$. We choose the measure $dx$ giving volume $1$ to $U/\Gamma^*$.
We write $v(\Phi,\c,dx)=v(\Phi,\c)$.

For the complex number $\zeta$, define ${\rm Todd}(\zeta,z)$ as
the expansion of
$$\frac{z}{1-\zeta^{-1} e^{-z}}$$
into a power series in $z$.
If $\zeta\neq 1$,  \begin{equation}\label{divi}
\frac{{\rm Todd}(\zeta,z)}{z}= \frac{1}{1-\zeta^{-1} e^{-z}}
\end{equation}
is analytic at $z=0$.

For $\phi\in \Phi,$ ${\rm Todd}(\zeta,\partial(\phi))$ is a
differential operator of infinite order with constant coefficients.
If  $p$ is a polynomial  function on $V$, ${\rm
Todd}(\zeta,\partial(\phi))p$ is well defined and is a polynomial on
$V$.

For $g\in T=U/\Gamma^*$,
define the {\sl Todd operator} (a series of differential operators
with constant coefficients) by
$${\rm Todd}(g,\Phi,\partial):=\prod_{k=1}^N
{\rm Todd}(e_g(\phi_k),\partial(\phi_k)),$$ where
$e_g(\phi_k):=e^{2i\pi \langle g,\phi_k\rangle}$.

If $G$ is a finite subset of $T$,  we denote by ${\rm
Todd}(G,\Phi,\partial)$ the operator defined on polynomial functions
$v(a)$ on $V$ by
$$({\rm Todd}(G,\Phi,\partial)v)(a)=\sum_{g\in G}
e_g(a)({\rm Todd}(g,\Phi,\partial)v)(a).$$

Let $g\in T$, and define $\Phi(g)=\{\phi\in \Phi\, |\, e_g(\phi)=1\}$.
If $\Phi(g)$ do not generate $V$, it follows from  Corollary \ref{cor:dah} that
$(\prod_{\phi\notin \Phi(g)} \partial(\phi)) v(\Phi,\c)=0$.
Thus ${\rm Todd}(g,\Phi,\partial)v(\Phi,\c)=0$. Indeed  ${\rm Todd}(g,\Phi,\partial)$
is divisible by  $(\prod_{\phi\notin \Phi(g)} \partial(\phi))$ as follows from Equation (\ref{divi}) above.
\begin{definition}
Define  $$G(\Phi)=\{g\in T\, |\, <\Phi(g)>= V\}.$$
\end{definition}
The set $G(\Phi)$ is finite. Indeed, if $g\in G(\Phi)$, there must
exists a basis $\sigma$ of $V$ extracted from $\Phi$ such that
$e_\phi(g)=1$ for all $\phi\in \sigma$, and this gives  a finite
set of solutions. If $\Phi$ is unimodular, then $G(\Phi)$ is
reduced to the identity element.

The following result has been proven in \cite{BriVer97}.
\begin{theorem}\label{th:toddgeneral}
Let $\c$ be a chamber. Then $$k(\Phi,\c)(a)= {\rm
Todd}(G(\Phi),\Phi,\partial)\cdot v(\Phi,\c).$$
\end{theorem}

Here we will give yet another proof of this theorem, by verifying
that the explicit formula for the jumps are related by the Todd
operator.

For the proof, it is easier to sum over `all elements'  $t$ of $T$.
If $v$ is a polynomial function on $V$ such that
\begin{equation}\label{H}
{\rm Todd}(t,\Phi,\partial)\cdot v=0 \mbox{ except for a finite number of elements $t$,}
\end{equation}
we may define
$${\rm
Todd}(T,\Phi,\partial)v(a)=\sum_{t\in T}e_t(a){\rm
Todd}(t,\Phi,\partial)\cdot v, $$ being understood that we only sum
over the finite subset of $t\in T$ such that ${\rm
Todd}(t,\Phi,\partial)\cdot v\neq 0$. With this definition, for
$v=v(\Phi,\c)$, then
$${\rm
Todd}(T,\Phi,\partial)v={\rm
Todd}(G(\Phi),\Phi,\partial)v .$$

To prove Theorem \ref{th:toddgeneral}, we follow the same scheme of
proof as in Theorem \ref{todduni}.

Let $W$ be a wall and  let $\c_{0}$ be a chamber of the wall for the
sequence $\Phi_0=\Phi\cap W$. We only need to prove
\begin{theorem}\label{toddgen}
\begin{equation}\label{eqKP}
{\rm Todd}(T,\Phi,\partial){\rm Pol}(v(\Phi_0,\c_0),\Phi\setminus
\Phi_0,E)={\rm Para}(k(\Phi_0,\c_0), \Phi\setminus \Phi_0,E).
\end{equation}
\end{theorem}
\begin{proof}
It is easy to see that the function $v:={\rm Pol}(v(\Phi_0,\c_0),\Phi\setminus
\Phi_0,E) $ satisfies the hypothesis (\ref{H}) above.
Let $V_0$ be a polynomial function on $V$
extending $v_0=v(\Phi_0, \c_0)$. Then
$({\rm Todd}(T,\Phi,\partial)v)(a)$ is equal to
$$
\sum_{t_0\in T_0}\sum_{t\in T\,|\, r(t)=t_0}e_t(a)({\rm
Todd}(t_0,\Phi_0,\partial) {\rm
Todd}(t,\Phi\setminus\Phi_0,\partial)\cdot v)(a),$$ where $T_0$
denotes the torus $W^*/\Gamma_0^*$. We have
$$e_t(a)({\rm Todd}(t,\Phi\setminus\Phi_0,\partial)v)(a)=\Res_{z=0}
\left(V_0(\partial_x)\cdot
\frac{e^{\langle a,x+2i\pi t+ z E\rangle }}{\prod_{\phi\in
\Phi\setminus \Phi_0}(1-e^{-\langle \phi,x+2i\pi t+zE\rangle}) }\right)_{x=0},$$
so that $({\rm Todd}(T,\Phi,\partial)v)(a)$ is equal to
$$=\sum_{t_0\in T_0}\sum_{t\in T\,| r(t)=t_0}\Res_{z=0}
\left(({\rm Todd}(t_0,\Phi_0,\partial)V_0)(\partial_x)
\frac{e^{\langle a,x+2i\pi t+ z E\rangle }}{\prod_{\phi\in
\Phi\setminus \Phi_0}(1-e^{-\langle \phi,x+2i\pi t+zE\rangle}) }\right)_{x=0}.$$

By induction hypothesis, a quasi-polynomial function extending
$k_0(\Phi_0,\c_0)$ is $$\sum_{t_0\in T_0}e_{t_0}(a) ({\rm
Todd}(t_0,\Phi_0,\partial)V_0)(a)$$ where we denote again by  $t_0$
any element of $T$ such that $r(t)=t_0$. Thus the last formula is
exactly the definition of ${\rm
Para}(k_0(\Phi_0,\c_0),\Phi\setminus\Phi_0,E)(a).$

The rest of the proof is identical to the proof of Theorem \ref{todduni}.
\end{proof}\bigskip

\section{Some examples}

In this section we give further examples of jumps in partition and
volume for various root systems.

\begin{example} We will calculate the jump in volume from
$\c_1:=C(\{e_1,e_1-e_2,e_1-e_3\}$ to
$\c_2:=C(\{e_1,e_1-e_2,e_1+e_3\})$ of $B_3$ (see Figure \ref{b3}
where chambers of $B_3$ are depicted via the intersection of the
ray $\R^+\alpha$  of each root with the plane $3a_1+2a_2+a_3=1$).
\begin{figure}
\begin{center}
  \includegraphics[width=60mm]{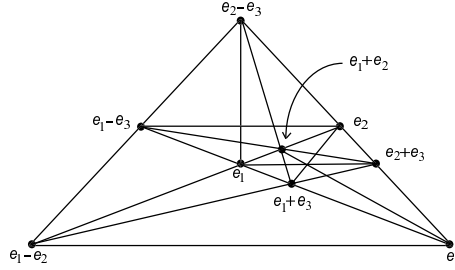}\\
  \caption{Chambers of $B_3$}\label{b3}
\end{center}
\end{figure}

We consider the copy $B_2$ in $B_3$ having the set
$\{e_1,e_2,e_1+e_2,e_1-e_2\}$ as its set of positive roots. This
particular jump is over the chamber $C(\{e_1,e_1-e_2\})$ of this
$B_2$. We have $E=e^3$, $\Phi^+=\{e_3,e_1+e_3,e_2+e_3\}$ and
$\Phi^-=\{e_1-e_3, e_2-e_3\}$. Using part (iii) of
Example~\ref{b2vol}, $v_{12}(a)=\frac{1}{4}(a_1+a_2)^2$. By
Theorem~\ref{jumpvol},
\[\begin{array}{l}
v(\Phi,\c_2)(a)-v(\Phi,\c_1)(a)={\rm Pol}(v_{12},\Phi\setminus \Phi_0,E)(a)\\
=\Res_{z=0} \left(v_{12}(\partial_{x}) \cdot
\frac{e^{a_1x_1+a_2x_2+a_3(z+x_3)}}{(x_3+z)(x_1+x_3+z)(x_2+x_3+z)(x_1-x_3-z)(x_2-x_3-z)}\right)_{x=0}\\
= \frac{1}{1440}a_3^4(30a_1a_2+15a_1^2+15a_2^2+2a_3^2).
\end{array}\]
\end{example}

\begin{example} We will calculate the jump in volume from
$\c_2=C(\{e_1,e_1-e_2,e_1+e_3\})$ to
$\c_3:=C(\{e_1,e_1+e_3,e_2-e_3\})\cap C(\{e_1,e_1+e_3,e_2+e_3\})$ of
$B_3$ (see Figure~\ref{b3}). We consider the copy $B_2$ in $B_3$
having the set $\{e_1,e_3,e_1+e_3,e_1-e_3\}$ as its set of positive
roots. This particular jump is over chamber $C(\{e_1,e_1+e_3\})$ of
this $B_2$. We have $E=e^2$,
$\Phi^+=\{e_2,e_2+e_3,e_2-e_3,e_1+e_2\}$ and $\Phi^-=\{e_1-e_2\}$.
Using part (ii) of Example~\ref{b2vol},
$v_{23}(a)=\frac{1}{4}(a_1+a_3)^2-\frac{1}{2}a_3^2$.  By
Theorem~\ref{jumpvol},
\[\begin{array}{ll}
v(\Phi,\c_3)(a)-v(\Phi,\c_2)(a)&={\rm Pol}(v_{23},\Phi\setminus \Phi_0,E)(a)\\
&=\Res_{z=0} \left(v_{23}(\partial_{x}) \cdot
\frac{e^{a_1x_1+a_2z+a_3x_3}}{z(x_3+z)(-x_3+z)(x_1+z)(x_1-z)}\right)_{x=0}\\
&= -\frac{1}{96}a_2^4(a_1^2+2a_1a_3-a_3^2)
\end{array}\]
\end{example}

\begin{example} We will calculate the jump in the partition function from
$\c_1=C(\{e_1,e_1-e_2,e_1-e_3\}$ to
$\c_2=C(\{e_1,e_1-e_2,e_1+e_3\})$ of $B_3$ (see Figure ~\ref{b3}).
This particular jump is over the chamber $C(\{e_1,e_1-e_2\})$ of the
copy of $B_2$ in $B_3$ having positive roots
$\Phi_0=\{e_1,e_2,e_1+e_2,e_1-e_2\}$. We have $E=e^3$,
$\Phi^+=\{e_3,e_1+e_3,e_2+e_3\}$, $\Phi^-=\{e_1-e_3, e_2-e_3\}$.

Using part (iii) of Example~\ref{kb2},
$k_{12}(a)=\frac{1}{4}(a_1+a_2)^2+a_1+a_2+\frac{7}{8}+(-1)^{a_1+a_2}\frac{1}{8}$.
Then, with the notation of Section 6, we have $y=0$ or
$y=\frac{e^1+e^2}{2}$; correspondingly $Q_0(a)=
\frac{1}{4}(a_1+a_2)^2+a_1+a_2+\frac{7}{8}$ and
$Q_{\frac{e^1+e^2}{2}}(a)=\frac{1}{8}$.  For $y=0$, the only
feasible $g \in T$ giving a nontrivial residue at $z=0$ for a
summand in ${\rm Para}(k_{12},\Phi\setminus\Phi_0,E)$ and satisfying
$r(g)=0$ is $g=0$. On the other hand, for $y=\frac{e^1+e^2}{2}$ the
feasible set of $g$ giving a nontrivial residue at $z=0$ for a
summand in ${\rm Para}(k_{12},\Phi\setminus\Phi_0,E)$ and satisfying
$r(g)=\frac{e^1+e^2}{2}$ is $\{\frac{e^1+e^2}{2},
\frac{e^1+e^2+e^3}{2}\}$. Then, by Theorem \ref{jumpkgen},

\[\begin{array}{l}
k(\Phi,\c_2)(a)-k(\Phi,\c_1)(a)={\rm Para}(k_{12},\Phi\setminus \Phi_0,E)(a)\\
=\Res_{z=0} \left(Q_0(\partial_{x})\cdot \frac{e^{\langle a,x
+ze^3\rangle}}{\prod_{\phi\in \Phi^+\cup \Phi^-}(1-e^{-\langle
\phi,x+ze^3\rangle})}\right)_{x=0}\\
+\Res_{z=0} \left(Q_{\frac{e^1+e^2}{2}}(\partial_{x})\cdot
\frac{e^{\langle a,x +i\pi(e^1+e^2)+ze^3\rangle }}{\prod_{\phi\in
\Phi^+\cup\Phi^-}(1-e^{-\langle \phi,x+i\pi(e^1+e^2)+ze^3\rangle })}\right)_{x=0}\\
+\Res_{z=0} \left(Q_{\frac{e^1+e^2}{2}}(\partial_{x})\cdot
\frac{e^{\langle a,x +i\pi(e^1+e^2+e^3)+ze^3\rangle
}}{\prod_{\phi\in \Phi^+\cup\Phi^-}(1-e^{-\langle \phi,x+i\pi(e^1+e^2+e^3)+ze^3\rangle })}\right)_{x=0}\\
=\Res_{z=0} \left(Q_0(\partial_{x})\cdot
\frac{e^{a_1x_1+a_2x_2+a_3z}}{(1-e^{-z})(1-e^{-(x_1+z)})(1-e^{-(x_2+z)})
(1-e^{-(x_1-z)})(1-e^{-(x_2-z)})}\right)_{x=0}\\
+\frac{1}{8}\Res_{z=0} \left(
\frac{(-1)^{a_1+a_2}e^{a_3z}}{(1-e^{-z})(1+e^{-z})^2(1+e^{z})^2}\right)+\frac{1}{8}\Res_{z=0}
\left( \frac{(-1)^{a_1+a_2+a_3}e^{a_3z}}{(1+e^{-z})(1-e^{-z})^2
(1-e^{z})^2}\right)\\
=\frac{1}{2880}a_3(a_3-1)(a_3+2)(a_3+1)(4a_3^2+4a_3+30a_1^2+60a_2a_1+30a_2^2+441+240a_1+240a_2)\\
+(-1)^{a_1+a_2}\frac{1}{128}+(-1)^{a_1+a_2+a_3}\frac{1}{384}(2a_3+1)(2a_3^2+2a_3-3).
\end{array}\]

Let $\gamma_3:=\frac{1-(-1)^{a_3}}{2}$ and
$\gamma_{12}:=\frac{1-(-1)^{a_1+a_2}}{2}$.  Then, after some
calculation, we can factor $k(\Phi,\c_2)(a)-k(\Phi,\c_1)(a)$ as:
\[\frac{1}{2880}(a_3-\gamma_3)(a_3+2-\gamma_3)\\
\cdot
\left((1-\gamma_3)(f_1-30(1-\gamma_{12})(1-2a_3))+\gamma_3(f_2-30(1-\gamma_{12})(3+2a_3))\right),\]
where
\[\begin{array}{ll}
f_1=&4a_3^4+4a_3^3+30a_1^2a_3^2+60a_2a_1a_3^2+240a_2a_3^2+30a_2^2a_3^2+437a_3^2+240a_3^2a_1\\
    &-34a_3-426-60a_2a_1-30a_1^2-240a_2-30a_2^2-240a_1\\
f_2=&4a_3^4+12a_3^3+30a_1^2a_3^2+60a_2a_1a_3^2+240a_2a_3^2+30a_2^2a_3^2+449a_3^2+240a_3^2a_1\\
  &+60a_2^2a_3+60a_1^2a_3+480a_3a_2+120a_2a_1a_3+912a_3+480a_1a_3+45
\end{array}\]
\end{example}

\begin{example} Let $\c_{\rm nice}$ denote the interior of the cone generated
by the roots $\{e_i-e_{r+1},\;1\leq i \leq r \}$ of $A_r$. With
the notation of Section \ref{sec:jumpuni} , $a=\sum_{i=1}^r
a_i(e_i-e_{r+1})$ is in $\c_{\rm nice}$ if and only if $a_i>0$ for
all $1\leq i \leq r$. Then, the copy of $A_{r-1}$ (with positive
roots $\{e_i-e_j: 2\leq i<j \leq r+1\}$) in $A_r$ can be thought
as the hyperplane $W$ with $\c_1=c_{\rm ext}$, $\c_2=\c_{\rm
nice}(A_{r})$ and $\c_{12}=\c_{\rm nice}(A_{r-1})$.  Together with
the fact that $k(A_{r-1},\c_{\rm nice})(a)$ is independent of
$a_{r-1}$, we have by Theorem~\ref{jumpkuni},
\[\begin{array}{ll}
k(A_r,\c_{{\rm nice}})(a)&={\rm Par}(k(A_{r-1},\c_{\rm nice}),\{
e_1-e_2,\ldots,e_1-e_{r+1}\},e^1)(a)\\
&=\Res_{z=0}\left(k(A_{r-1},\c_{\rm nice})(\partial_{x_2},\dots,
\partial_{x_{r-1}}) \cdot
\frac{e^{a_1z+a_2x_2+\cdots+a_{r-1}x_{r-1}}}{(1-e^{x_2-z})
\cdots(1-e^{x_{r-1}-z})(1-e^{-z})^2}\right)_{x=0}.
\end{array}\]

For example, using $k(A_2,\c_{\rm nice})(a)=a_1+1$ (part (i) of
Example~\ref{ka2}),
\[
\begin{array}{ll}
k(A_3,\c_{{\rm nice}})(a)&=\Res_{z=0}\left((\partial_{x_2}+1) \cdot
\frac{e^{a_1z+a_2x_2}}{(1-e^{-z})^2(1-e^{x_2-z})}\right)_{x=0}\\
&=\frac{1}{6}(a_1+2)(a_1+1)(a_1+3a_2+3).
\end{array}\]
 We can iteratively calculate,
\[
\begin{array}{ll}
k(A_4,\c_{{\rm nice}})(a)&=\Res_{z=0}\left(
\frac{1}{6}(\partial_{x_2}+2)(\partial_{x_2}+1)(\partial_{x_2}+3\partial_{x_3}+3)\cdot
\frac{e^{a_1z+a_2x_2+a_3x_3}}{(1-e^{-z})^2(1-e^{x_2-z})(1-e^{x_3-z})}\right)_{x=0}\\
&=\frac{1}{360}(a_1+3)(a_1+2)(a_1+1)(a_1+3+a_2+3a_3)\\
&.(a_1^2+9a_1+5a_1a_2+10a_2^2+20+30a_2).
\end{array}\]
\end{example}

In a similar fashion, using Theorem \ref{jumpvol}, we can calculate
$v(A_n, c_{\rm nice})$ iteratively. The computation of $v(A_7,
c_{{\rm nice}})$ took 12 seconds. The result is too big to be
written here.

Recall that Baldoni-Beck-Cochet-Vergne \cite{bbcv} can compute
individual numbers $k(A_n)(a)$ for $a$ fixed, for $n=10$ in less
than $30$ minutes. The full polynomial  $k(A_n,c_{{\rm nice}})$ is
computed in $7$ minutes  when $n=7$ and $30$ minutes   when $n=8$
on a $1,13$GHz computer. The method of Baldoni-Beck-Cochet-Vergne
uses an arbitrary order on roots. The method of calculation which
follows from wall crossing formulae seems   less efficient, but it
may give some light on the best order strategy and the complexity
of calculations.

\end{document}